\documentclass[a4paper,12pt]{article}
\usepackage[english]{babel}
\usepackage[utf8]{inputenc}
\usepackage{geometry}
\usepackage{hyperref}
\usepackage{amssymb}
\usepackage{amsmath}
\usepackage{amsthm}
\allowdisplaybreaks
\title{Coisotropic Hypersurfaces in Grassmannians}
\author{Kathlén Kohn\\Institute of Mathematics,\\ TU Berlin, Germany\\
\rm{kohn@math.tu-berlin.de}}
\date{}

\theoremstyle{definition}
\newtheorem{defn}{Definition}
\newtheorem{ex}[defn]{Example}

\theoremstyle{plain}
\newtheorem{thm}[defn]{Theorem}
\newtheorem{lem}[defn]{Lemma}
\newtheorem{prop}[defn]{Proposition}
\newtheorem{cor}[defn]{Corollary}

\theoremstyle{remark}
\newtheorem{rem}[defn]{Remark}

\newcommand{\bs}[1]{\boldsymbol{#1}}
\newcommand{\proj}{\mathrm{proj}}
\newcommand{\cone}{\mathrm{cone}}
\newcommand{\tr}{\mathrm{tr}}
\newcommand{\rs}{\mathrm{rs}}
\newcommand{\pl}{\mathrm{pl}}
\newcommand{\Hom}{\mathrm{Hom}}

\newcommand{\im}{\mathrm{im}\,}
\newcommand{\Gr}{\mathrm{Gr}}
\newcommand{\CH}{\mathrm{CH}}
\newcommand{\PP}{\mathbb{P}}
\newcommand{\codim}{\mathrm{codim}\,}

\begin{document}
\maketitle

\begin{abstract}
	To every projective variety $X$, we associate a list of hypersurfaces in different Grassmannians, called the coisotropic hypersurfaces of $X$. These include the Chow form and the Hurwitz form of $X$. Gel'fand, Kapranov and Zelevinsky characterized coisotropic hypersurfaces by a rank one condition on tangent spaces. We present a new and simplified proof of that result. We show that the coisotropic hypersurfaces of $X$ equal those of its projectively dual variety, and that their degrees are the polar degrees of $X$. Coisotropic hypersurfaces of Segre varieties are defined by hyperdeterminants, and all hyperdeterminants arise in that manner.  We generalize Cayley's differential characterization of coisotropy and derive new equations for the Cayley variety which parametrizes all coisotropic hypersurfaces of given degree in a fixed Grassmannian. We provide a \texttt{Macaulay2} package for transitioning between $X$ and its coisotropic hypersurfaces.
	\\ \hspace*{2mm} 
\end{abstract}

\section{Chow Forms, Hurwitz Forms and Beyond}

Historically one of the main motivations for the introduction of Chow forms was the parametrization of all subvarieties of $\mathbb{P}^n$ with fixed dimension and degree.
The parametrization of 0-dimensional varieties is trivial, hypersurfaces can be parametrized by their defining equations, and linear subspaces are points in their respective Grassmannian.
Hence, the first non-trivial case are curves in 3-space of degree (at least) two.
Chow forms of space curves were first introduced by Cayley \cite{cayley2}.
The generalization to arbitrary varieties was given by Chow and van der Waerden \cite{chow}.
For a given variety $X \subseteq \mathbb{P}^n$ of dimension $k$, projective subspaces of dimension $n-k-1$ have typically no intersection with the variety, but those subspaces that do intersect $X$ form a hypersurface in the corresponding Grassmannian.
The \emph{Chow form} of $X$ is the defining polynomial of this hypersurface, which is a unique (up to scaling with constants) polynomial in the coordinate ring of the Grassmannian.
It has the same degree as $X$ and determines $X$ uniquely.

Consider now the same given variety $X$, but projective subspaces of dimension $n-k$. These are expected to have a finite number of intersection points with $X$, given by the degree $d$ of $X$. The subspaces which do not intersect $X$ in $d$ reduced points form again a hypersurface in the corresponding Grassmannian, as long as $d \geq 2$. The defining polynomial of this hypersurface is called \emph{Hurwitz form} of $X$~\cite{hurwitz}.
Equivalently, this hypersurface is the Zariski closure of the set of all $(n-k)$-subspaces that intersect $X$ non-transversely at some smooth point.
In this sense, it plays a crucial role in the study of the condition of intersecting a projective variety with varying linear subspaces \cite{condition}.
Moreover, Chow forms of curves and Hurwitz forms of surfaces in 3-space define exactly the self-dual hypersurfaces in the Grassmannian of lines in $\mathbb{P}^3$ \cite{catanese}.

A natural generalization of the above described hypersurfaces is studied in Chapters~3.3 and~4.3 of \cite{gkz}: 
the \emph{$i$-th higher associated hypersurface} of a variety~$X \subseteq \PP^n$ of dimension $k$ is defined as the Zariski closure of the set of all $(n-k-1+i)$-dimensional subspaces that intersect $X$ at a smooth point non-transversely.
This article is devoted to the study of these hypersurfaces.
Our main contributions are the following.
We state for which index $i$ the above definition yields indeed a hypersurface (Cor.~\ref{cor:dim}), and we show that the degrees of these hypersurfaces are the polar degrees of $X$ (Thm.~\ref{thm:degrees}).
We give a new proof for the main result in \cite{gkz} about these hypersurfaces:
the higher associated hypersurfaces are exactly the \emph{coisotropic hypersurfaces}~\cite[Ch.~4, Thm.~3.14]{gkz}, which are hypersurfaces in Grassmannians whose normal spaces at smooth points are spanned by homomorphisms of rank one.
Note that the latter definition is independent of any underlying projective variety.
The proof in \cite{gkz} uses machinery from Lagrangian varieties to prove the existence of the projective variety $X$ a given coisotropic hypersurface is associated to, whereas our proof (Thm.~\ref{thm:coisotropy_equiv}) is more direct and defines the underlying variety $X$ explicitly.
Due to this result, we will use a more meaningful name for our main objects of study: 
in the following, the $i$-th higher associated hypersurface of $X$ is called the \emph{$i$-th coisotropic hypersurface} of $X$.
Furthermore, we derive that the coisotropic hypersurfaces of the projectively dual variety~$X^\vee$ are the coisotropic hypersurfaces of $X$ in reversed order (Thm.~\ref{thm:duality}).
Finally, we generalize Cayley's~\cite{cayley2} differential characterization of coisotropy (Thm.~\ref{thm:plueckerChar}), and present a \texttt{Macaulay2}~\cite{m2} package for computations with coisotropic hypersurfaces (Sec.~\ref{sec:computations}).

Coisotropic hypersurfaces have many applications in well-studied areas as well as recent research.
As we will see in Section~\ref{sec:degree}, coisotropic hypersurfaces are the analogue of polar varieties in Grassmannians.
Hyperdeterminants are special cases of coisotropic hypersurfaces (see Sec.~\ref{sec:hyperdet}).
Moreover, coisotropic hypersurfaces and their (iterated) singular loci appear naturally in algebraic vision~\cite{vision}.
The unitary group acts transitively on the tangent spaces of a coisotropic hypersurface, which allows to compute volumes of coisotropic hypersurfaces with the kinematic formula~\cite{schubert}.
This is essential for the probabilistic Schubert calculus proposed in~\cite{schubert}.

Section~\ref{sec:dim} shows that the coisotropic hypersurfaces of a projective variety are essentially projectively dual to the Segre product of the variety with projective spaces. With this, we determine for which indices $i$ the $i$-th coisotropic hypersurface has indeed codimension one in its ambient Grassmannian.
We prove in Section~\ref{sec:degree} that the degrees of the coisotropic hypersurfaces coincide with the polar degrees. Section~\ref{sec:rank1} revisits the main theorem of \cite{gkz} about coisotropic hypersurfaces.
The coisotropic hypersurfaces of a variety and its projectively dual will be related in Section~\ref{sec:dual}.
We discuss hyperdeterminants in Section~\ref{sec:hyperdet}.
The Cayley variety~-- formed by all coisotropic hypersurfaces of fixed degree in a fixed Grassmannian~-- as well as characterizations for coisotropy in Pl\"ucker coordinates are studied in Section~\ref{sec:pluecker}. Finally, we shortly describe our \texttt{Macaulay2} package in Section~\ref{sec:computations}.
As a preparation, naming conventions for different coordinate systems of Grassmannians and some further basic notions will be introduced in the following section.

\section{Preliminaries}

This article focuses on projective varieties over $\mathbb{C}$.
Let $V$ be an $(n+1)$-dimensional complex vector space, and $\mathbb{P}(V)$ its projectivization.
If $V = \mathbb{C}^{n+1}$, we will simply write $\mathbb{P}^n := \mathbb{P}(\mathbb{C}^{n+1})$.
By $\mathbb{P}(V)^\ast$ we denote the projectivization of the dual space $V^\ast$, which is formed by hyperplanes in $\mathbb{P}(V)$. The two projective spaces $\mathbb{P}(V)^\ast$ and $\mathbb{P}(V)$ can be identified by sending every point $y=(y_0 : \ldots : y_n) \in \mathbb{P}(V)$ to the hyperplane
\begin{align}
\label{eq:identifyHPpoint}
	\left\lbrace y = 0 \right\rbrace := \left\lbrace x \in \mathbb{P}(V) \;\middle\vert\; \sum \limits_{i=0}^n y_i x_i = 0 \right\rbrace \in \mathbb{P}(V)^\ast.
\end{align}

Throughout the rest of this article, let $X \subseteq \mathbb{P}(V)$ be an irreducible non-empty variety.
We denote the set of smooth points of $X$ by $\mathrm{Reg}(X)$, and
the \emph{embedded tangent space} of $X$ at $x \in \mathrm{Reg}(X)$ by 
\begin{align*}
	T_xX := \left\lbrace y \in \mathbb{P}(V) \;\middle\vert\; \forall f \in I(X): \sum \limits_{i=0}^n \frac{\partial f}{\partial X_i}(x) \cdot y_i = 0 \right\rbrace,
\end{align*}
where $I(X) \subseteq \mathbb{C}[X_0, \ldots, X_n]$ denotes the vanishing ideal of $X$.
A hyperplane in $\mathbb{P}(V)$ is called \emph{tangent} to $X$ at a smooth point $x \in X$ if it contains $T_xX$.
The Zariski closure $X^\vee$ in $\mathbb{P}(V)^\ast$ of the set of all hyperplanes that are tangent to $X$ at some smooth point is called the \emph{projectively dual variety} of $X$.
If $X$ is irreducible, so is $X^\vee$ \cite[Ch.~1, Prop.~1.3]{gkz}.
We will make frequent use of the following biduality of projective varieties over $\mathbb{C}$, which is also known as reflexivity.
\begin{thm}[{\cite[Ch.~1, Thm.~1.1]{gkz}}]
	For every projective variety $X \subseteq \mathbb{P}(V)$, we have $(X^\vee)^\vee = X$. More precisely, if $x \in \mathrm{Reg}(X)$ and $H \in \mathrm{Reg}(X^\vee)$, then $H$ is tangent to $X$ at $x$ if and only if $x$~-- regarded as a hyperplane in $\mathbb{P}(V)^\ast$~-- is tangent to $X^\vee$ at $H$. 
\end{thm}

The Grassmannian of all projective subspaces in $\mathbb{P}^n$ of dimension $l$ is denoted \linebreak[4] by~\mbox{$\Gr(l,\PP^n)$}.
For $0 \leq i \leq k := \dim(X)$ and a projective subspace $L \subseteq \mathbb{P}^n$ of dimension $n-k+i-1$, the dimension of $L$ intersected with $T_xX$ equals $i-1$ for almost all $x \in X$. Those subspaces that have a larger intersection with some $T_xX$ (for $x \in \mathrm{Reg}(X)$) form a subvariety of the respective Grassmannian.
\begin{defn}
\label{defn:higher_associated}
For $i \in \lbrace0, \ldots, k \rbrace$, the \emph{$i$-th coisotropic variety of $X$} is defined as
\begin{align*}
\CH_i(X) := \overline{\left\lbrace L \mid \exists x \in \mathrm{Reg}(X) \cap L : \dim(L \cap T_x X) \geq i \right\rbrace} \subseteq \Gr(n-k+i-1,\PP^n).
\end{align*}
\end{defn}
In \cite{gkz} these varieties are called \emph{higher associated hypersurfaces} (cf. Ch.~3, Sec.~2E and Ch.~4, Sec.~3).
For $i=0$, the above definition reduces to the \emph{Chow hypersurface}
\begin{align*}
	\CH_0(X) = \left\lbrace L \in \Gr(n-k-1,\PP^n) \mid X \cap L \neq \emptyset \right\rbrace.
\end{align*}
We will see in Corollary~\ref{cor:dim} that $\CH_0(X)$ is indeed a hypersurface in $\Gr(n-k-1,\PP^n)$ for every $X$.
Thus it is defined by one polynomial in the coordinate ring of $\Gr(n-k-1,\PP^n)$, which is unique up to a constant factor \cite[Ch.~3, Prop.~2.1]{gkz}, known as the \emph{Chow form} of $X$.
The case $i=1$ is studied in~\cite{hurwitz}.
The variety $\CH_1(X)$ is a hypersurface if and only if $\deg(X) \geq 2$, and its corresponding polynomial in the coordinate ring of $\Gr(n-k,\PP^n)$ is called the \emph{Hurwitz form} of~$X$.
For all $i$, the condition $\dim(L \cap T_x X) \geq i$ is equivalent to $\dim(L + T_x X) \leq n-1$, meaning that $L$ intersects $X$ at $x$ non-transversely.

\subsection{Coordinate Systems}

Different coordinates on Grassmannians are discussed in \cite[Ch.~3, Sec.~1]{gkz}. We follow the conventions used in~\cite{hurwitz}. There are six different ways of specifying a point $L \in \Gr(l,\PP^n)$, which we call \emph{primal/dual Pl\"ucker/Stiefel/affine coordinates}.

First, let $A \in \mathbb{C}^{(n-l) \times (n+1)}$ be such that $L$ is the projectivization of the kernel of $A$. The entries of $A$ are the \emph{primal Stiefel coordinates} of $L$, and the maximal minors $p_{i_1 \ldots i_{n-l}}$ of $A$ are the \emph{primal Pl\"ucker coordinates} of $L$. Up to scaling, the Pl\"ucker coordinates are unique, whereas the Stiefel coordinates are clearly not:
multiplying $A$ with any invertible $(n-l) \times (n-l)$ matrix does not change its kernel.
Hence, when denoting by $S(n-l,n+1)$ the \emph{Stiefel variety} of all complex $(n-l) \times (n+1)$ matrices of full rank, the Grassmannian $\Gr(l,\PP^n)$ can be seen as the quotient $S(n-l,n+1) / GL(n-l)$.

Pick now a maximal linearly independent subset of columns of $A$, indexed \linebreak[4] by $\lbrace i_1,\ldots, i_{n-l}\rbrace$, and multiply the inverse of this submatrix by $A$ itself. The resulting matrix has the same kernel as $A$ and its columns indexed by $\lbrace i_1,\ldots, i_{n-l}\rbrace$ form the identity matrix. The remaining entries of this new matrix are the \emph{primal affine coordinates} of $L$. These give a unique representation of $L$ in the \emph{primal affine chart} 
\begin{align*}
U_{i_1 \ldots i_{n-l}} := \left\lbrace A \in \mathbb{C}^{(n-l) \times (n+1)} \;\middle\vert\; 
\begin{array}{l}
i_j\text{-th column of } A \text{ is standard basis vector }e_j \\ (\text{for } j=1, \ldots, n-l)
\end{array}
\right\rbrace
\end{align*}
of the Grassmannian $\Gr(l,\PP^n)$.

Secondly, let $B \in \mathbb{C}^{(l+1) \times (n+1)}$ such that $L$ is the projectivization of the row space of~$B$. The entries of $B$ are the \emph{dual Stiefel coordinates} of $L$, and the maximal minors $q_{j_0 \ldots j_{l}}$ of~$B$ are the \emph{dual Pl\"ucker coordinates} of $L$. As above, for every maximal linearly independent subset of columns of $B$, indexed by $\lbrace j_0,\ldots, j_{l}\rbrace$, the subspace $L$ has unique \emph{dual affine coordinates} in the \emph{dual affine chart} $U_{j_0 \ldots j_{l}}$
 of the Grassmannian $\Gr(l,\PP^n)$.
 
We use the identification $\Gr(n-l-1,(\PP^n)^\ast) = \Gr(l,\PP^n)$ given by the canonical isomorphism of both Grassmannians.
Moreover, we fix the non-canonical isomorphism between $\Gr(l,\PP^n)$ and $\Gr(n-l-1,\PP^n)$ that sends a linear subspace $L \in \Gr(l,\PP^n)$ to its orthogonal complement $L^\perp \in \Gr(n-l-1,\PP^n)$ with respect to the non-degenerate bilinear form $(x,y) \mapsto \sum x_i y_i$ (and \emph{not} with respect to the complex scalar product $(x,y) \mapsto \sum x_i \overline{y_i}$).

The primal coordinates of $\Gr(l,\PP^n)$ are the dual coordinates of $\Gr(n-l-1,\PP^n)$. 
If $L \in \Gr(l,\PP^n)$ is the row space of $B := (I_{l+1} | M)$, where $I_k$ denotes the $k$-dimensional identity matrix and $M$ is an $(l+1) \times (n-l)$-matrix, then $L^\perp$ is the row space of $A := (-M^T | I_{n-l})$. Moreover, the kernel of $A$ is $L$ and the kernel of $B$ is $L^\perp$, since we have 
\begin{align*}
	\ker (N) = \rs(N)^\perp
\end{align*}
for any matrix $N$ with row space $\rs(N)$. Thus the isomorphism between $\Gr(l,\PP^n)$ and $\Gr(n-l-1,\PP^n)$ on the same coordinates in both Grassmannians coincides with the map between primal and dual coordinates within the same Grassmannian.
For the Pl\"ucker coordinates it follows that 
\begin{align}
\label{eq:plucker_change}
q_{j_0 \ldots j_{l}} = s(i_1, \ldots, i_{n-l}) \cdot p_{i_1 \ldots i_{n-l}},
\end{align}
where $i_1, \ldots, i_{n-l}$ form the complement of $\lbrace j_0, \ldots, j_{l} \rbrace$ in strictly increasing order, and $s(i_1, \ldots, i_{n-l})$ denotes the sign of the permutation $(i_1, \ldots, i_{n-l},j_0, \ldots, j_{l})$.

\begin{ex}
\label{Ex:runningPlueckerEqations}
Consider the following polynomial in the primal Pl\"ucker coordinates of $\Gr(1,\PP^3)$, which will appear again in Examples~\ref{ex:runningFermatCubic} and~\ref{ex:runningFermatDual}:
\begin{align}
\label{eq:hurwitzFermat}
\begin{split}
&(p_{01}^6+p_{02}^6+p_{03}^6+p_{12}^6+p_{13}^6+p_{23}^6)\\
+2&(p_{10}^3p_{02}^3+p_{10}^3p_{03}^3+p_{20}^3p_{03}^3
+p_{01}^3p_{12}^3+p_{01}^3p_{13}^3+p_{21}^3p_{13}^3)\\
+2&(p_{02}^3p_{21}^3+p_{02}^3p_{23}^3+p_{12}^3p_{23}^3
+p_{03}^3p_{31}^3+p_{03}^3p_{32}^3+p_{13}^3p_{32}^3)\\
+2 &\left( p_{01}p_{23}(p_{03}^2p_{12}^2-p_{02}^2p_{13}^2)-p_{02}p_{13}(p_{01}^2p_{23}^2-p_{03}^2p_{12}^2)+p_{03}p_{12}(p_{02}^2p_{13}^2-p_{01}^2p_{23}^2) \right).
\end{split}
\end{align}
To display the symmetry of the polynomial, the convention $p_{ji}=-p_{ij}$ for $i<j$ is used.
The change of coordinates
\begin{align*}
p_{01} \mapsto q_{23},p_{02} \mapsto -q_{13},
p_{03} \mapsto q_{12},p_{12} \mapsto q_{03},
p_{13} \mapsto -q_{02},p_{23} \mapsto q_{01},
\end{align*}
yields the polynomial in dual Pl\"ucker coordinates:
\begin{align}
\label{eq:hurwitzDualFermat}
\begin{split}
&(q_{01}^6+q_{02}^6+q_{03}^6+q_{12}^6+q_{13}^6+q_{23}^6)\\
-2&(q_{10}^3q_{02}^3+q_{10}^3q_{03}^3+q_{20}^3q_{03}^3
+q_{01}^3q_{12}^3+q_{01}^3q_{13}^3+q_{21}^3q_{13}^3)\\
-2&(q_{02}^3q_{21}^3+q_{02}^3q_{23}^3+q_{12}^3q_{23}^3
+q_{03}^3q_{31}^3+q_{03}^3q_{32}^3+q_{13}^3q_{32}^3)\\
+2 &\left( q_{01}q_{23}(q_{03}^2q_{12}^2-q_{02}^2q_{13}^2)-q_{02}q_{13}(q_{01}^2q_{23}^2-q_{03}^2q_{12}^2)+q_{03}q_{12}(q_{02}^2q_{13}^2-q_{01}^2q_{23}^2) \right).
\end{split}
\end{align}
The polynomial in primal or dual Stiefel coordinates is obtained by substituting the $2 \times 2$-minor given by the columns $i$ and $j$ of a general $2 \times 4$-matrix
$\left( \begin{smallmatrix} 
a_{11} & a_{12} & a_{13} & a_{14} \\ a_{21} & a_{22} & a_{23} & a_{24}
\end{smallmatrix} \right)$ 
into $p_{ij}$ or~$q_{ij}$, respectively. Similarly, one gets the polynomial in primal or dual affine coordinates by using a matrix of the form $\left( \begin{smallmatrix} 
1 & 0 & a_{13} & a_{14} \\ 0 & 1 & a_{23} & a_{24}
\end{smallmatrix} \right)$, i.e., by substituting
$$ \quad\quad\quad
p_{01} \mapsto 1,p_{02} \mapsto a_{23},
p_{03} \mapsto a_{24},p_{12} \mapsto -a_{13},
p_{13} \mapsto -a_{14},p_{23} \mapsto a_{13}a_{24}-a_{23}a_{14}.
\hskip \textwidth minus \textwidth \diamondsuit
$$
\end{ex}

For a subvariety $\Sigma \subseteq \Gr(l,\PP^n)$, we denote by $\Sigma^\ast \subseteq \Gr(n-l-1,\PP^n)$ the image under the non-canonical isomorphism defined above.
In particular, for $l=0$ and $X \subseteq \PP^n$, we have that $X^\ast \subseteq (\PP^n)^\ast$ is the image under the identification described in~\eqref{eq:identifyHPpoint}.

\begin{rem}
To a projective subspace $L \subseteq \PP^n$ we have associated three different subspaces:
its orthogonal complement $L^\perp \subseteq \PP^n$, 
the set $L^\vee \subseteq (\PP^n)^\ast$ of hyperplanes in $\PP^n$ containing $L$,
and the set $L^\ast \subseteq (\PP^n)^\ast$ of hyperplanes in $\PP^n$ parametrized by points in $L$ as in~\eqref{eq:identifyHPpoint}. \hfill $\diamondsuit$
\end{rem}

\section{Cayley Trick}
\label{sec:dim}

The Chow hypersurface can be constructed as the dual of a Segre product, which is known as the \emph{Cayley trick} \cite[Ch.~3, Thm.~2.7]{gkz}.
We generalize this for all coisotropic varieties. Consider the Segre embedding 
\begin{align*}
	\mathbb{P}^{k-i} \times X \hookrightarrow  \mathbb{P}^{k-i} \times \mathbb{P}^n &\hookrightarrow \mathbb{P} \left(\mathbb{C}^{(k-i+1) \times (n+1)} \right), \text{ and} \\
	\left( (\mathbb{P}^{k-i} \times X)^\vee \right)^\ast &\hookrightarrow \mathbb{P} \left(\mathbb{C}^{(k-i+1) \times (n+1)} \right).
\end{align*}
With this, we will show that the projectively dual variety $(\mathbb{P}^{k-i} \times X)^\vee$ equals $\CH_i(X)$, when both varieties are interpreted in the primal matrix space $\mathbb{P} (\mathbb{C}^{(k-i+1) \times (n+1)} )$. Formally, consider the following projection from the Stiefel variety $S(k-i+1,n+1)$ of all complex $(k-i+1) \times (n+1)$-matrices of full rank onto the ambient Grassmannian of $\CH_i(X)$:
\begin{align*}
p: S(k-i+1,n+1) &\longrightarrow \Gr(n-k+i-1,\PP^n), \\
A &\longmapsto \proj(\ker A),
\end{align*}
and take the closure $\overline{p^{-1}(\CH_i(X))}$ in the primal matrix space $\mathbb{P} (\mathbb{C}^{(k-i+1) \times (n+1)} )$.

\begin{prop}
\label{prop:cayley_trick}
The varieties $\overline{p^{-1}(\CH_i(X))}$ and $((\mathbb{P}^{k-i} \times X)^\vee)^\ast$ in the matrix space \linebreak[4] $\mathbb{P} (\mathbb{C}^{(k-i+1) \times (n+1)} )$ are equal.
\end{prop}

\begin{proof}
This proof follows the lines of the proof of \cite[Ch.~3, Thm.~2.7]{gkz}.

The variety $((\mathbb{P}^{k-i} \times X)^\vee)^\ast$ is the Zariski closure of the set of all~$A \in \mathbb{P} (\mathbb{C}^{(k-i+1) \times (n+1)} )$ such that the hyperplane 
\begin{align*}
	\lbrace A=0 \rbrace := \left\lbrace M \in \mathbb{P} (\mathbb{C}^{(k-i+1) \times (n+1)}) \;\middle\vert\; \sum_{i,j} a_{ij} m_{ij} = 0 \right\rbrace
\end{align*}
is tangent to $\mathbb{P}^{k-i} \times X$ at some point $(y,x)$ with $x \in \mathrm{Reg}(X)$. Denote by $\bs{x} \in \mathbb{C}^{n+1}$, $\bs{y} \in \mathbb{C}^{k-i+1}$ and $\bs{A} \in \mathbb{C}^{(k-i+1) \times (n+1)}$ affine representatives of $x$, $y$ and $A$, respectively. Moreover, let $\cone(X)$ be the affine cone over $X$. Then we have
\begin{align*}
T_{(\bs{y},\bs{x})} \cone(\mathbb{P}^{k-i} \times X)
&= \lbrace (\bs{y} \otimes v) + (w \otimes \bs{x}) \mid v \in T_{\bs{x}} \cone(X), w \in \mathbb{C}^{k-i+1} \rbrace \\
&= (\bs{y} \otimes T_{\bs{x}} \cone(X)) + (\mathbb{C}^{k-i+1} \otimes \bs{x}) \subseteq \mathbb{C}^{(k-i+1) \times (n+1)}.
\end{align*}

Now consider $\overline{p^{-1}(\CH_i(X))}$, which is the Zariski closure of the set of all full rank $A \in \mathbb{P} (\mathbb{C}^{(k-i+1) \times (n+1)} )$ such that $\bs{x} \in \ker \bs{A}$ and $\dim(\ker(\bs{A}) \cap T_{\bs{x}} \cone(X)) \geq i+1$ for some $x \in \mathrm{Reg}(X)$. Hence, it is enough the show the following two equivalences: 
\begin{align*}
\lbrace \bs{A}=0 \rbrace \supseteq \mathbb{C}^{k-i+1} \otimes \bs{x}
&\Leftrightarrow
\bs{x} \in \ker \bs{A}, \\
\exists \bs{y} \in \mathbb{C}^{k-i+1} \setminus \lbrace 0 \rbrace: \lbrace \bs{A}=0 \rbrace \supseteq \bs{y} \otimes T_{\bs{x}} \cone(X)
&\Leftrightarrow
\dim(\ker(\bs{A}) \cap T_{\bs{x}} \cone(X)) \geq i+1.
\end{align*}

Due to $\ker \bs{A} = \lbrace u \in \mathbb{C}^{n+1} \mid \forall w \in \mathbb{C}^{k-i+1}: \sum_{i,j} a_{ij} w_i u_j =0 \rbrace$, the first equivalence is obvious. Consider the linear map $\psi: T_{\bs{x}} \cone(X) \to \mathbb{C}^{k-i+1}, v \mapsto \bs{A}v$ for the second equivalence. Since $\dim (T_{\bs{x}} \cone(X))=k+1$, the dimension of $ \ker \psi = \ker(\bs{A}) \cap T_{\bs{x}} \cone(X) $ is at least $ i+1$ if and only if $\psi$ is not surjective, which is equivalent to the existence of some non-zero $\bs{y}\in \mathbb{C}^{k-i+1}$ orthogonal to $\im \psi$, i.e., $\sum_{i,j} a_{ij}y_iv_j=0$ for all $v \in T_{\bs{x}} \cone(X)$.
\end{proof}

The construction in Proposition \ref{prop:cayley_trick} will be the main ingredient for many of the following proofs. Furthermore, it shows that the defining polynomials of the coisotropic varieties (in case they are all hypersurfaces) interpolate from the Chow form via the Hurwitz form to the \emph{$X$-discriminant} which is the defining polynomial of $X^\vee$.

This raises immediately the next question: when is the $i$-th coisotropic variety indeed a hypersurface in the Grassmannian.

\begin{cor}
\label{cor:dim}
$\quad \CH_i(X)$ has codimension one in $\Gr(n-k+i-1,\PP^n)$ if and only if $i \leq k - \codim X^\vee +1$.
\end{cor}

\begin{proof}
For any irreducible variety $X \subset \PP^n$, we define $\mu(X) := \dim X + \codim X^\vee -1$.
As shown in \cite[Ch.~1, Thm~5.5,]{gkz}, we have 
\begin{align}
\label{eq:product_thm}
\mu(X \times Y) = \max \lbrace \dim X + \dim Y, \mu(X), \mu(Y) \rbrace
\end{align}
for the product $X \times Y \hookrightarrow \mathbb{P}^{(n+1)(m+1)-1}$ of two irreducible varieties $X \subseteq \PP^n$ and $Y \subseteq \PP^m$.
Hence, the dual $(X \times Y)^\vee$ is a hypersurface in $(\mathbb{P}^{(n+1)(m+1)-1})^\ast$ if and only if the above maximum equals $\dim X + \dim Y$.
This holds in particular when one of the factors is the variety $\mathbb{P}^{k-i}$ embedded into itself such that $(\mathbb{P}^{k-i})^\vee = \emptyset$ \cite[Ch.~1, Cor.~5.9]{gkz}.
By convention, $\dim (\mathbb{P}^{k-i})^\vee=-1$, so $\mu (\mathbb{P}^{k-i}) = 2(k-i)$. Thus $(\mathbb{P}^{k-i} \times X)^\vee$ is a hypersurface if and only if 
$2k-i \geq  k+\codim X^\vee-1$.
\end{proof}

This motivates the following definition.

\begin{defn}
For $i \in \lbrace 0, \ldots, k-\codim X^\vee+1 \rbrace$, the hypersurface $\CH_i(X)$ is called the \emph{$i$-th coisotropic hypersurface of $X$}. Its defining polynomial in the coordinate ring of $\Gr(n-k+i-1,\PP^n)$, which is unique up to a constant factor, is called the \emph{$i$-th coisotropic form of $X$}.
\end{defn}

\begin{ex}
\label{ex:runningFermatCubic}
Let $X \subseteq \mathbb{P}^3$ be the surface defined by the Fermat cubic \mbox{$x_0^3+x_1^3+x_2^3+x_3^3$}. The projectively dual of $X$ is also a surface. Therefore, the surface $X$ has three coisotropic hypersurfaces. The Chow form of $X$ in dual Pl\"ucker coordinates of $\Gr(0,\PP^3) = \mathbb{P}^3$ is just the Fermat cubic itself. The Hurwitz form of $X$ in primal and dual Pl\"ucker coordinates of $\Gr(1,\PP^3)$ is given by the polynomials in Example~\ref{Ex:runningPlueckerEqations}. This was computed with \texttt{Macaulay2}. Finally, the second coisotropic form of $X$ in primal Pl\"ucker coordinates $p_i$ of $\Gr(2,\PP^3)$, which are the dual coordinates of $\mathbb{P}^3$, is the following polynomial of degree 12, which is also the defining equation of $(X^\vee)^\ast$:
\begin{align*}
6(z_0^4+z_1^4+z_2^4+z_3^4)-8&(z_0^3+z_1^3+z_2^3+z_3^3)
(z_0+z_1+z_2+z_3)\\
+(z_0^2+z_1^2+z_2^2+z_3^2)^2+2&(z_0^2+z_1^2+z_2^2+z_3^2)(z_0+z_1+z_2+z_3)^2-40z_0z_1z_2z_3,
\end{align*}
where $z_i := p_i^3$ for $0 \leq i \leq 3$. \hfill $\diamondsuit$
\end{ex}

\section{Polar Degrees}
\label{sec:degree}

After studying the dimension of the coisotropic hypersurfaces, the next focus will lie on their degrees. In fact, these degrees agree with the well-studied polar degrees \cite{piene, holme}. 
As before, let $X \subseteq \mathbb{P}^n$ be an irreducible variety of dimension $k$.
Moreover, let $0 \leq i \leq k$ and $V \subseteq \mathbb{P}^n$ be a projective subspace of dimension $n-k+i-2$.
For almost all $x \in X$, the dimension of $V$ intersected with $T_xX$ equals $i-2$. 
Define the \emph{$i$-th polar variety of $X$ with respect to $V$} as
\begin{align*}
P_i(X,V) := \overline{\left\lbrace x \in \mathrm{Reg}(X) \mid \dim(V \cap T_x X) \geq i-1 \right\rbrace} \subseteq X.
\end{align*}
Given a general $X$, the $i$-th polar variety has codimension $i$ in $X$ for almost all choices of $V$. Furthermore, for any $X$ there exists an integer $\delta_i(X)$ that is equal to the degree of $P_i(X,V)$ for almost all $V$. 
This integer $\delta_i(X)$ is called the \emph{$i$-th polar degree of $X$}.

These degrees satisfy a lot of interesting properties:
\hspace*{-5mm}
\begin{enumerate}
\item $\delta_i(X) > 0$ if and only if $i \leq k - \codim X^\vee +1$.

(Note that this coincides with the range of indices where the coisotropic varieties of~$X$ are hypersurfaces.)
\item $\delta_0(X) = \deg X$.
\item $\delta_{k - \codim X^\vee +1}(X) = \deg X^\vee$.
\item $\delta_i(X) = \delta_{k - \codim X^\vee +1-i}(X^\vee)$.
\item $\delta_i(X \cap H) = \delta_i(X)$ for any $0 \leq i \leq k-1$ and any generic hyperplane $H \subseteq \mathbb{P}^n$.
\item $\delta_i(\pi(X)) = \delta_i(X)$ if $\codim X \geq 2$ and $\pi: \mathbb{P}^n \dashrightarrow \mathbb{P}^{n-1}$ is a general linear projection.
\end{enumerate}

One can also define the polar degrees via the  \emph{conormal variety}
\begin{align*}
	\mathcal{N}_X := \overline{\left\lbrace(x,y) \mid x \in \mathrm{Reg}(X), T_x X \subseteq \lbrace y= 0 \rbrace \right\rbrace} \subseteq \mathbb{P}^n \times \mathbb{P}^n.	
\end{align*}
The \emph{multidegree} of a variety $X$ embedded into a product of projective spaces $\PP^{n_1} \times \ldots \times \PP^{n_d}$ with codimension $c$ is a homogeneous polynomial of degree $c$ whose term $k t_1^{c_1} \ldots t_d^{c_d}$ indicates that the intersection of $X$ with the product $L_1 \times \ldots \times L_d$ of general linear subspaces $L_i \subseteq \PP^{n_i}$ with $\dim(L_i) = c_i$ consists of $k$ points.
Thus, the {multidegree} of $\mathcal{N}_X$ is a homogeneous polynomial of degree $n+1$ in two variables.
The non-zero coefficients of this polynomial are the polar degrees (cf. \cite[Prop. (3) on page 187]{kleiman1} and \cite[Lem. (2.23) on page 169]{kleiman2}).
Using the command \texttt{multidegree} in \texttt{Macaulay2}, this gives a practical way to compute the polar degrees of a given variety $X$.

Now another property will be added to this list, namely that the degree of the $i$-th coisotropic hypersurface of $X$ is the $i$-th polar degree of $X$. 
On first sight, this is remarkable since the coisotropic hypersurfaces are subvarieties of a Grassmannian, whereas the polar varieties are subvarieties of the projective variety $X \subseteq \mathbb{P}^n$. 
The degree of a hypersurface $\Sigma \subseteq \Gr(l,\PP^n)$ is defined as
\begin{align*}
\deg \Sigma := |\lbrace L \in \Sigma \mid N \subseteq L \subseteq M \rbrace|,
\end{align*}
where $N \subseteq M \subseteq \mathbb{P}^n$ is a generic flag of $(l-1)$-dimensional and $(l+1)$-dimensional projective subspaces.
Alternatively, the degree of $\Sigma$ can be defined as the degree of the defining polynomial of $\Sigma$ in the coordinate ring of $\Gr(l,\PP^n)$~\cite[Ch.~3, Prop.~2.1]{gkz}.

\begin{thm}
	\label{thm:degrees}
For $0 \leq i \leq k - \codim X^\vee +1$, the degree of the $i$-th coisotropic hypersurface of $X$ equals the $i$-th polar degree of $X$, i.e.,
\begin{align*}
\deg \CH_i(X) = \delta_i(X).
\end{align*}
\end{thm}

\begin{proof}
	Let $0 \leq d \leq k$. For $0 \leq i \leq k-d$ and a generic subspace $M \subseteq \mathbb{P}^n$ of codimension~$d$, we have $\delta_i(X \cap M) = \delta_i(X)$ by applying the fifth property above several times.
	Fix now $0 \leq i \leq  k-\codim X^\vee+1$ and set $d := k-i$. Choose a generic $M$ of codimension $d$ as well as a generic subspace $N \subseteq M$ with $\dim N = \dim M-\dim (X \cap M)+i-2 = n-k+i-2$. Then the $i$-th polar degree of $X$ equals $\deg P_i(X \cap M,N)$. Since $X \cap M$ is $i$-dimensional, it follows from the first and the third property above that the dual $(X \cap M)^\vee$ is a hypersurface in $M^\ast$ with degree $\delta_i(X \cap M)$. To sum up,
	\begin{align*}
		\delta_i(X) = \deg P_i(X \cap M,N) = \delta_i(X \cap M) = \deg (X \cap M)^\vee.
	\end{align*}
	The degree of the hypersurface $(X \cap M)^\vee$ is also the number of hyperplanes in $M$ that are tangent to some smooth point of $X \cap M$ and that contain $N$, but these hyperplanes are exactly the subspaces in $\CH_i(X)$ with $N \subseteq L \subseteq M$.
\end{proof}

\begin{rem}
	For general projective varieties $X$, we can give another geometric argument to show Theorem~\ref{thm:degrees}.
	As above, let $V$ be a generic projective subspace of dimension $n-k+i-2$. Consider the variety $S_i(X,V) \subseteq \mathbb{P}^n$ formed by the union of all lines through $V$ and the $i$-th polar variety $P_i(X,V)$. For general $X$, the $i$-th polar variety has codimension $i$ in $X$ and $S_i(X,V)$ is a hypersurface of degree $\delta_i(X)$.
	The $i$-th coisotropic form of $X$ in dual Stiefel coordinates is a polynomial in the entries of a general $(n-k+i)\times(n+1)$-matrix $B$. Substituting the last rows of that matrix by a basis of $V$ yields a homogeneous polynomial $F \in \mathbb{C}[b_{0j} \mid 0 \leq j \leq n ]$, whose degree is the degree of the $i$-th coisotropic hypersurface of~$X$. This polynomial defines an irreducible hypersurface in $\mathbb{P}^n$, which is in fact $S_i(X,V)$. This shows Theorem~\ref{thm:degrees} for general $X$.
	
	To see that $S_i(X,V)$ and the zero locus of $F$ are the same, it is enough to show that~$F$ vanishes at every point in $S_i(X,V)$.
	This is clear for all points in $V \subseteq S_i(X,V)$. For a point $y \notin V$ on the line between $x \in P_i(X,V)$ and some point in $V$ we have that $F$ vanishes at $y$ if and only if it vanishes at $x$. If $x$ is a smooth point of $X$ such that the dimension of $V \cap T_xX$ is at least $i-1$, then the projective span of $V$ and $x$ is a point in $\CH_i(X)$ and $F$ vanishes at $x$. Since the set of all those $x$ is dense in $P_i(X,V)$, all points in $P_i(X,V)$ are in the zero locus of $F$. \hfill $\diamondsuit$
\end{rem}

\section{Rank One Characterizations}
\label{sec:rank1}

In the following, several equivalent characterizations for coisotropic hypersurfaces will be given. A very fundamental equivalence statement is proven in \cite[Ch.~4, Thm.~3.14]{gkz}. This proof contains many geometric ideas by taking a detour over conormal varieties and Lagrangian varieties in general. Here a new and direct proof will be presented, using only the Cayley trick from Proposition \ref{prop:cayley_trick}.
In~\cite{gkz}, the notion of coisotropy is defined as follows:

\begin{defn}
Let $U$ and $V$ be finite dimensional vector spaces. Define $W := \Hom(U,V)$, and identify $W^\ast$ with $\Hom(V,U)$ via $\Hom(V,U) \ni \phi \mapsto \tr(\cdot \circ \phi)$. A hyperplane $H \subseteq W$ is called \emph{coisotropic} if its defining equation has rank one in $\Hom(V,U)$.
\end{defn}

In coordinates, we choose bases for $U^\ast$ and $V$, and let $E_{ij} \in U^\ast \otimes V = \Hom(U,V)$ be the matrix with exactly one 1-entry at position $(i,j)$ and 0-entries otherwise.
Let $\psi \in W^\ast$ be the equation for the hyperplane $H$.
We consider the bases for $U$ and $V^\ast$ dual to the chosen bases above,
and define $N_\psi \in V^\ast \otimes U = \Hom(V,U)$ as the matrix whose $(j,i)$-th entry equals~$\psi(E_{ij})$.
The coisotropy condition for $H$ means exactly that the rank of $N_\psi$ equals one.
This definition can be extended to hypersurfaces in Grassmannians, since the tangent space of $\Gr(l,\PP^n)$ at $L=\PP(U)$, for an $(l+1)$-dimensional subspace $U \subseteq \mathbb{C}^{n+1}$, is naturally isomorphic to $\Hom(U,\mathbb{C}^{n+1}/U)$.

\begin{defn}
\label{defn:coisotropy}
An irreducible hypersurface $\Sigma$ in the Grassmannian $\Gr(l,\PP^n)$ is called \emph{coisotropic} if, for every smooth $L \in \Sigma$, the tangent hyperplane $T_L \Sigma$ is coisotropic in $T_L \Gr(l,\PP^n)$.
\end{defn}

Let $L \in \Gr(l,\PP^n)$ be given in primal affine coordinates by an $(n-l) \times (l+1)$-matrix $M_L$.
This means that for some fixed $i_1, \ldots, i_{n-l}$, the subspace $L$ is the projectivization of the kernel of an $(n-l) \times (n+1)$-matrix $A$ whose columns indexed by $i_1, \ldots, i_{n-l}$ are standard basis vectors and the remaining columns form $M_L$.
Let $\rho := \rho_{i_1, \ldots, i_{n-l}}$ be the map that sends a given matrix $M_L$ to its corresponding primal Pl\"ucker coordinates, which are the maximal minors of $A$.
Thus, if $Q$ denotes the polynomial in primal Pl\"ucker coordinates that defines $\Sigma$, then $Q \circ \rho$ is the equation for $\Sigma$ in the primal affine chart $U_{i_1, \ldots, i_{n-l}}$. Identifying the tangent hyperplane of $\Sigma$ at $L$ with
\begin{align*}
T_{M_L} V(Q \circ \rho) = \left\lbrace
{M} \in \mathbb{C}^{(n-l) \times (l+1)} \;\middle\vert\; \sum \limits_{i,j} \frac{\partial (Q \circ \rho)}{\partial a_{ij}} (M_L) \cdot {M}_{ij} =0
\right\rbrace
\end{align*}
yields that $T_L \Sigma$ is coisotropic if and only if the rank of the matrix $J_{Q \circ \rho}(M_L)$ with entries $\frac{\partial (Q \circ \rho)}{\partial a_{ij}} (M_L)$ is one.

\begin{defn}
For a polynomial $F \in \mathbb{C}[a_{ij} \,|\, 1 \leq i \leq l, 1 \leq j \leq m]$, let $J_F$ be the $(l \times m)$-matrix of all first-order partial derivatives of $F$, i.e., the $(i,j)$-th entry of $J_F$ is~$\frac{\partial F}{\partial a_{ij}}$.
\end{defn}

Theorem 3.14 in Chapter 4 of \cite{gkz} states that the two notions of coisotropy given in Definitions \ref{defn:higher_associated} and \ref{defn:coisotropy} are in fact the same:

\begin{thm}
\label{thm:coisotropy_equiv}
\hspace*{3mm}
\begin{enumerate}
\item For an irreducible variety $X \subseteq \mathbb{P}^n$ and $i \in \lbrace 0, \ldots, \dim X-\codim X^\vee+1 \rbrace$, the $i$-th coisotropic hypersurface of $X$ is coisotropic (in the sense of Definition~\ref{defn:coisotropy}).
\item For an irreducible coisotropic hypersurface $\Sigma \subseteq \Gr(l,\PP^n)$, there is an irreducible variety $X \subseteq \mathbb{P}^n$ such that $\Sigma = \CH_{\dim X+l+1-n}(X)$.
\end{enumerate}
\end{thm}

This is remarkable since Definition \ref{defn:coisotropy} does not depend on the underlying projective variety $X$. Before proving the theorem, some helpful and easy characterizations of coisotropy will be established. The first one says that it is enough to check coisotropy for a hypersurface in the Grassmannian on \emph{one fixed affine chart} of the Grassmannian to deduce coisotropy for the whole hypersurface. This statement can also be found as Proposition 3.12 in Chapter 4 of \cite{gkz}.

\begin{prop}
\label{prop:affine_characterization}
Let $\Sigma \subseteq \Gr(l,\PP^n)$ be an irreducible hypersurface, given by a homogeneous polynomial $Q$ in primal Pl\"ucker coordinates.
Moreover, fix a primal affine chart $U_{i_1, \ldots, i_{n-l}}$ of $\Gr(l,\PP^n)$ together with the map $\rho := \rho_{i_1, \ldots, i_{n-l}}$ sending $(n-l) \times (l+1)$-matrices to their corresponding primal Pl\"ucker coordinates.
Then $\Sigma$ is coisotropic (in the sense of Definition \ref{defn:coisotropy}) if and only if the $2 \times 2$-minors of $J_{Q \circ \rho}(M)$ are zero for all $M \in V(Q \circ \rho)$.
\end{prop}

\begin{proof}
Assume first that $\Sigma$ is coisotropic. For every smooth point $L \in \Sigma$ representable in the affine chart $U_{i_1, \ldots, i_{n-l}}$ by the $(n-l)\times(l+1)$-matrix $M_L$, we have that the rank of $J_{Q \circ \rho}(M_L)$ is one. For all singular points this rank is zero. This proves one direction.

For the other direction, assume that $J_{Q \circ \rho}(M)$ has rank at most one for all $(n-l)\times(l+1)$-matrices $M \in V(Q \circ \rho)$, i.e., that $J_{Q \circ \rho}(M)$ has rank exactly one for all $M \in \mathrm{Reg}(V(Q \circ \rho))$. 
If $Q=p_{i_1 \ldots i_{n-l}}$, then one can directly calculate that the coisotropy condition is also satisfied on all other affine charts; hence $\Sigma$ is coisotropic.
Otherwise, let $L \in \mathrm{Reg}(\Sigma)$ be representable in the affine chart $U_{i'_1, \ldots, i'_{n-l}}$ by the $(n-l)\times(l+1)$-matrix $M_L$, and let $\rho' := \rho_{i'_1, \ldots, i'_{n-l}}$ be the corresponding Pl\"ucker map, where $(i'_1, \ldots, i'_{n-l})$ is different from $(i_1, \ldots, i_{n-l})$. It is left to show that the rank of $J_{Q \circ \rho'}(M_L)$ is one.
Since the smooth points in $\Sigma$ form an open subset of $\Sigma$, there is an open neighborhood~$\mathcal{U}$ of $M_L$ such that all $M \in \mathcal{U} \cap V(Q \circ \rho')$ correspond to smooth points in $\Sigma$. Due to the assumption that $Q \neq p_{i_1 \ldots i_{n-l}}$, every open neighborhood $\mathcal{V} \subseteq \mathcal{U}$ of $M_L$ contains a point $M \in \mathcal{V} \cap V(Q \circ \rho')$ whose corresponding smooth point in $\Sigma$ can be represented in the affine chart $U_{i_1, \ldots, i_{n-l}}$.
By assumption, $J_{Q \circ \rho'}(M)$ has rank one for all those $M$. It follows from the continuity of the $2 \times 2$-minors of $J_{Q \circ \rho'}$ that the rank of $J_{Q \circ \rho'}(M_L)$ is also one.
\end{proof}

It is sometimes more practical to work in Stiefel coordinates than in an affine chart. For example, the defining polynomial of a hypersurface in a Grassmannian is still homogeneous when written in Stiefel coordinates, but generally not in affine coordinates. Furthermore, the Cayley trick from Proposition \ref{prop:cayley_trick} uses Stiefel coordinates. To give the coisotropic characterization in Stiefel coordinates (Proposition \ref{prop:stiefel_characterization}), a technical lemma is needed.

\begin{lem}
\label{lem:jacobi_mult}
Denote by $\pl$ the map that sends a given matrix to its maximal minors.
Let $Q$ be a homogeneous polynomial in Pl\"ucker coordinates of $\Gr(l,\PP^n)$, and let $N \in \mathbb{C}^{(n-l) \times (n+1)}$ as well as $U \in \mathbb{C}^{(n-l) \times (n-l)}$ with $\det(U) \neq 0$. Then we have
\begin{align}
\label{eq:jacobi_mult}
J_{Q \circ \pl}(UN) = \det(U)^{\mathrm{deg}(Q)-1} \mathrm{adj}(U)^T J_{Q \circ \pl}(N).
\end{align}
\end{lem}

\begin{proof}
For $m \in \mathbb{N}$, let $[m] := \lbrace 1, \ldots, m \rbrace$.
Note first that
\begin{align*}
\frac{\partial \det}{\partial a_{ij}}(A) = (-1)^{i+j} \det(A_{[n-l]\setminus \lbrace i \rbrace,[n-l]\setminus \lbrace j \rbrace}) = \mathrm{adj}(A)_{ji},
\end{align*}
holds for all matrices $A \in \mathbb{C}^{(n-l) \times (n-l)}$. This identity implies
\begin{align*}
\frac{\partial \det}{\partial a_{ij}}(UA) &= \mathrm{adj}(UA)_{ji}
= \left( \mathrm{adj}(A) \mathrm{adj}(U) \right)_{ji}
= \sum \limits_{m=1}^{n-l} \mathrm{adj}(A)_{jm} \mathrm{adj}(U)_{mi} \\
&= (-1)^{j+i} \sum \limits_{m=1}^{n-l} \det(A_{[n-l]\setminus \lbrace m \rbrace,[n-l]\setminus \lbrace j \rbrace}) \det(U_{[n-l]\setminus \lbrace i \rbrace,[n-l]\setminus \lbrace m \rbrace}).
\end{align*}
For $I \subseteq [n+1]$ with $|I|=n-l$ and $j \in I$, let $\mathrm{pos}(j,I) \in [n-l]$ denote the position of $j \in I$ when $I$ is ordered strictly increasing (i.e., $j$ is the $\mathrm{pos}(j,I)$-th element of $I$). Then:
\begin{align*}
\frac{\partial (Q \circ \pl)}{\partial a_{ij}}(N) 
&= \sum \limits_I \frac{\partial Q}{\partial p_I} (\pl(N)) \frac{\partial p_I}{\partial a_{ij}}(N)\\
&= \sum \limits_{I: j \in I} \frac{\partial Q}{\partial p_I} (\pl(N)) (-1)^{i+\mathrm{pos}(j,I)} \det (N_{[n-l]\setminus \lbrace i \rbrace,I\setminus \lbrace j \rbrace}).
\end{align*}
Finally, set $d := \mathrm{deg}(Q)-1$ and show that the entries of the two matrices in \eqref{eq:jacobi_mult} are equal:
\begin{align*}
&\frac{\partial (Q \circ \pl)}{\partial a_{ij}}(UN) \\
=& \sum \limits_{I: j \in I} \frac{\partial Q}{\partial p_I} (\det(U)\pl(N)) (-1)^{i+\mathrm{pos}(j,I)} \sum \limits_{m=1}^{n-l} \det(U_{[n-l]\setminus \lbrace i \rbrace,[n-l]\setminus \lbrace m \rbrace}) \det(N_{[n-l]\setminus \lbrace m \rbrace,I\setminus \lbrace j \rbrace}) \\
= &\det(U)^d \sum \limits_{m=1}^{n-l} \det(U_{[n-l]\setminus \lbrace i \rbrace,[n-l]\setminus \lbrace m \rbrace}) \\
&\quad \cdot (-1)^{i+m}
\sum \limits_{I: j \in I} \frac{\partial Q}{\partial p_I} (\pl(N)) (-1)^{m+\mathrm{pos}(j,I)}\det(N_{[n-l]\setminus \lbrace m \rbrace,I\setminus \lbrace j \rbrace}) \\
= &\det(U)^d \sum \limits_{m=1}^{n-l} \det(U_{[n-l]\setminus \lbrace i \rbrace,[n-l]\setminus \lbrace m \rbrace}) (-1)^{i+m}
\frac{\partial (Q \circ \pl)}{\partial a_{mj}}(N) \\
= &\det(U)^d \sum \limits_{m=1}^{n-l} \mathrm{adj}(U)_{mi}
\frac{\partial (Q \circ \pl)}{\partial a_{mj}}(N).
\qedhere
\end{align*}
\end{proof}

\begin{prop}
\label{prop:stiefel_characterization}
Let $l \geq 1$, and let $\Sigma \subseteq \Gr(l,\PP^n)$ be an irreducible hypersurface, given by a homogeneous polynomial $Q$ in primal Pl\"ucker coordinates. Then~$\Sigma$ is coisotropic (in the sense of Definition \ref{defn:coisotropy}) if and only if the $2\times 2$-minors of $J_{Q \circ \pl}(N)$ are zero for all $N \in V(Q \circ \pl)$.
\end{prop}

\begin{proof}
First assume that the rank of $J_{Q \circ \pl}(N)$ is at most one for all $N \in V(Q \circ \pl)$. By Proposition \ref{prop:affine_characterization}, it is enough to consider the affine chart $U_{1, \ldots, n-l}$ together with the Pl\"ucker map $\rho := \rho_{1, \ldots, n-l}$. Hence, it is sufficient to show that the rank of $J_{Q \circ \rho}(M)$ is at most one when $M \in V(Q \circ \rho)$.
Indeed, $(I_{n-l} | M) \in V(Q \circ \pl)$ and the rank of $J_{Q \circ \pl}(I_{n-l} | M)$ is at most one by assumption. Since the last $l+1$ columns of $J_{Q \circ \pl}(I_{n-l} | M)$ coincide with~$J_{Q \circ \rho}(M)$, it follows that $J_{Q \circ \rho}(M)$ has rank at most one.

Secondly, let $\Sigma$ be coisotropic, and let $N \in V(Q \circ \pl)$. We may assume that all $(n-l)\times(n-l)$-minors of $N$ are non-zero, since this always holds up to a change of coordinates. We now fix a $2\times 2$-minor of $J_{Q \circ \pl}$ and show that it vanishes on $N$. Without loss of generality, this minor does not contain entries from the first $n-l$ columns of $J_{Q \circ \pl}$ (by the assumption $l \geq 1$). Hence, it is enough to show that the last $l+1$ columns of~$J_{Q \circ \pl}(N)$ form a matrix of rank at most one. Since the first $n-l$ columns of $N$ form an invertible matrix $U$, one can write $N = U \cdot (I_{n-l} | M)$ where $M$ is an $(n-l)\times (l+1)$-matrix. Moreover, $\rho := \rho_{1, \ldots, n-l} = \pl(I_{n-l}|\cdot)$ implies
\begin{align*}
\frac{\partial (Q \circ \rho)}{\partial m_{ij}}(M)
=\frac{\partial (Q \circ \pl(I_{n-l} | \cdot))}{\partial m_{ij}}(M)
=\frac{\partial (Q \circ \pl)}{\partial m_{ij}}(I_{n-l} | M).
\end{align*}
Together with Lemma \ref{lem:jacobi_mult} and $M \in V( Q \circ \rho )$, this leads to
\begin{align*}
\mathrm{rank} \left( J_{Q \circ \pl}(N)_{\lbrace n-l+1, \ldots, n+1 \rbrace} \right)
&= \mathrm{rank} \left( \det(U)^{\deg(Q)-1} \mathrm{adj}(U)^T J_{Q \circ \pl}(I_{n-l}|M)_{\lbrace n-l+1, \ldots, n+1 \rbrace} \right) \\
&\leq \mathrm{rank} \left( J_{Q \circ \pl}(I_{n-l}|M)_{\lbrace n-l+1, \ldots, n+1 \rbrace} \right) \\
&= \mathrm{rank} \left( J_{Q \circ \rho}(M) \right)\leq 1,
\end{align*}
where $A_{\lbrace n-l+1, \ldots, n+1 \rbrace}$ denotes the last $l+1$ columns of an $(n-l)\times(n+1)$-matrix $A$.
\end{proof}

The above proposition uses the assumption $l \geq 1$ to place a $2\times 2$-minor outside of a maximal square submatrix. Note that for $l=0$, every hypersurface in $\Gr(0,\PP^n) = \mathbb{P}^n$ is coisotropic. With this, it can be proven that the two given definitions of coisotropy coincide.

\begin{proof}[Proof of Theorem~\ref{thm:coisotropy_equiv}.]
For the first part, let $k := \dim X$ and let $Q$ be the defining polynomial of $\CH_i(X)$ in the primal Pl\"ucker coordinates of $\Gr(n-k+i-1,\PP^n)$.
The map $\pl$ sends a given matrix to its maximal minors.
By Proposition \ref{prop:cayley_trick}, the polynomial $Q \circ \pl$ is the defining equation of $((\mathbb{P}^{k-i} \times X)^\vee)^\ast$.
If $n-k+i \leq 1$ or $n-k+i \geq n$, then $\CH_i(X)$ is trivially coisotropic.
Thus assume $1 < n-k+i < n$.
By Proposition \ref{prop:stiefel_characterization}, it is enough to show that the rank of $J_{Q \circ \pl}(N)$ equals one for all $N \in \mathrm{Reg}(((\mathbb{P}^{k-i} \times X)^\vee)^\ast)$.
For such an~$N$, denote by $H_N$ the tangent hyperplane to $((\mathbb{P}^{k-i} \times X)^\vee)^\ast$ at $N$.
Then $H_N$ is a point in $(((\mathbb{P}^{k-i} \times X)^\vee)^\ast )^\vee = (\mathbb{P}^{k-i} \times X)^\ast$, by biduality. Since the defining equation for the hyperplane~$H_N$ is given by $J_{Q \circ \pl}(N)$, i.e.,
\begin{align*}
H_N = \left\lbrace M \in \mathbb{P}\left( \mathbb{C}^{(k-i+1) \times (n+1)} \right) \;\middle\vert\; \sum \limits_{i,j} \frac{\partial (Q \circ \pl)}{\partial a_{ij}} (N) \cdot M_{ij} = 0 \right\rbrace,
\end{align*}
the point in $\mathbb{P}^{k-i} \times X$ corresponding to $H_N$ is $J_{Q \circ \pl}(N)$. Hence, $J_{Q \circ \pl}(N) \in \mathbb{P}^{k-i} \times X$, which implies that the rank of $J_{Q \circ \pl}(N)$ is one.

Consider now the second part of the theorem.
Let again $Q$ be the defining polynomial of $\Sigma \subseteq \Gr(l,\PP^n)$ in primal Pl\"ucker coordinates, such that $Q \circ \pl$ is the defining equation in primal Stiefel coordinates.
If $l = 0$, then $\Sigma \subseteq \mathbb{P}^n$ and $\CH_0(\Sigma)=\Sigma$.
Therefore, assume that $l \geq 1$.
For the hypersurface $V(Q \circ \pl)$ we have~-- as above~-- that
\begin{align*}
\left(V(Q \circ \pl)^\vee\right)^\ast = \overline{\left\lbrace J_{Q \circ \pl}(N) \;\middle\vert\; N \in V(Q \circ \pl), \mathrm{rank} \left( J_{Q \circ \pl}(N) \right) \neq 0 \right\rbrace} \subseteq \PP \left( \mathbb{C}^{(n-l) \times (n+1)} \right).
\end{align*}
By Proposition \ref{prop:stiefel_characterization}, the rank of $J_{Q \circ \pl}(N) \in \mathbb{C}^{(n-l) \times (n+1)}$ is at most one for all $N \in V(Q \circ \pl)$, which implies that $(V(Q \circ \pl)^\vee)^\ast \subseteq  \mathbb{P}^{n-l-1} \times \mathbb{P}^n$.
For all smooth $N \in V(Q \circ \pl)$, denote by $x_N \in \mathbb{P}^n$ the projective point spanning the 1-dimensional affine row space of $J_{Q \circ \pl}(N)$.
This defines an irreducible projective variety
\begin{align*}
X := \overline{
\left\lbrace x_N \;\middle\vert\; N \in V(Q \circ \pl), \mathrm{rank} \left( J_{Q \circ \pl}(N) \right) \neq 0
\right\rbrace} \subseteq \mathbb{P}^n
\end{align*}
such that $(V(Q \circ \pl)^\vee)^\ast \subseteq \mathbb{P}^{n-l-1} \times X$.
In fact, equality holds: $(V(Q \circ \pl)^\vee)^\ast = \mathbb{P}^{n-l-1} \times X$.
To see this, consider $N \in \mathrm{Reg}(V(Q \circ \pl))$ and $y = (y_1 : \ldots : y_{n-l}) \in \mathbb{P}^{n-l-1}$.
Without loss of generality we can assume that $y_1=1$. Denote by $v \in \mathbb{C}^{n-l}$ a non-zero column of $J_{Q \circ \pl}(N)$. Then we have for all $0 \leq j \leq n$ that the $j$-th column of $J_{Q \circ \pl}(N)$ is a scalar multiple $\mu_j v$ such that $(\mu_0 : \ldots : \mu_n)=x_N$. Pick now a basis $(w^{(2)}, \ldots, w^{(n-l)})$ of $\lbrace w \in \mathbb{C}^{n-l} \mid \sum_{i=1}^{n-l} \overline{w_i} v_i = 0 \rbrace$. Let $a_1 := \overline{v}$ and $a_i := y_i \overline{v} + \overline{w^{(i)}}$ (for $2 \leq i \leq n-l$) be the rows of a matrix $A \in \mathbb{C}^{(n-l) \times (n-l)}$. Then $A$ is invertible and $A \cdot J_{Q \circ \pl}(N)$ equals $y \otimes x_N$ (up to scaling). The adjugate matrix $U := \mathrm{adj}(A^T)$ is also invertible, $UN$ is a smooth point of $V(Q \circ \pl)$, and by Lemma \ref{lem:jacobi_mult},
\begin{align*}
J_{Q \circ \pl}(UN) = \det(U)^{\mathrm{deg}(Q)-1} \det(A)^{n-l-2}A \cdot J_{Q \circ \pl}(N)
\end{align*}
is equal to $y \otimes x_N$ (up to scaling).
It follows that $V(Q \circ \pl) = ((\mathbb{P}^{n-l-1} \times X)^\vee)^\ast$, which concludes the proof using the Cayley trick in Proposition \ref{prop:cayley_trick}.
\end{proof}

Theorem~\ref{thm:coisotropy_equiv} and Propositions~\ref{prop:affine_characterization} and \ref{prop:stiefel_characterization} provide practical tools to check for coisotropy of a given hypersurface in a Grassmannian. Moreover, the Cayley trick in Proposition~\ref{prop:cayley_trick} and the proof of Theorem~\ref{thm:coisotropy_equiv} give ways to recover the underlying projective variety.

\begin{ex}
All coisotropic forms in the Grassmannian $\Gr(1,\PP^3)$ of lines in $\mathbb{P}^3$ are either Chow forms of curves or Hurwitz forms of surfaces.
According to~\cite{catanese}, these define exactly the self-dual hypersurfaces in $\Gr(1,\PP^3)$. \hfill $\diamondsuit$
\end{ex}

\begin{ex}
\label{ex:runningFermatDual}
As in Example~\ref{ex:runningFermatCubic}, let $X \subseteq \mathbb{P}^3$ be the surface defined by the Fermat cubic \linebreak[4] $x_0^3 + x_1^3 + x_2^3 + x_3^3$.
The polynomials \eqref{eq:hurwitzFermat} and \eqref{eq:hurwitzDualFermat} given in Example~\ref{Ex:runningPlueckerEqations} are the Hurwitz form of $X$ in primal and dual Pl\"ucker coordinates.
Thus both polynomials define the same hypersurface $\CH_1(X) \subseteq \Gr(1,\PP^3)$.
Now consider the hypersurface $\CH_1(X)^\ast$ in the same Grassmannian $\Gr(1,\PP^3)$, whose defining equation in primal Pl\"ucker coordinates is the second polynomial \eqref{eq:hurwitzDualFermat} and whose defining equation in dual Pl\"ucker coordinates is the first polynomial \eqref{eq:hurwitzFermat}.
Recall that $\CH_1(X)^\ast$ is geometrically obtained from $\CH_1(X)$ by sending every line $L \in \CH_1(X)$ to its orthogonal complement $L^\perp$.
Using the characterization of coisotropy in Proposition~\ref{prop:affine_characterization}, it follows that $\CH_1(X)^\ast$ is also coisotropic.
By Theorem~\ref{thm:coisotropy_equiv}, the hypersurface $\CH_1(X)^\ast$ has an underlying projective variety:
it turns out that $\CH_1(X)^\ast$ is the first coisotropic hypersurface of the surface $(X^\vee)^\ast$ dual to $X$, which has degree 12 and is given by the last polynomial in Example~\ref{ex:runningFermatCubic}.
Hence, the two similar polynomials of degree 6 from Example~\ref{Ex:runningPlueckerEqations} are the Hurwitz forms of a surface of degree 3 and a surface of degree 12.
This behavior will be investigated in general in the following section. \hfill $\diamondsuit$
\end{ex}

\section{Duality}
\label{sec:dual}

So far only characterizations of coisotropy using primal coordinates were treated.
The case of dual coordinates will be considered now by proceeding as in Example~\ref{ex:runningFermatDual}.
It follows from Proposition \ref{prop:affine_characterization} that a hypersurface $\Sigma \subseteq \Gr(1,\PP^n)$ is coisotropic if and only if $\Sigma^\ast \subseteq \Gr(n-l-1,\PP^n)$ is coisotropic.
Furthermore, since changing between the two above Grassmannians is the same as changing between primal and dual coordinates, it follows that Propositions \ref{prop:affine_characterization} and \ref{prop:stiefel_characterization} also hold for dual (instead of primal) coordinates.
This raises the question of how the underlying projective varieties of $\Sigma$ and $\Sigma^\ast$ are related.

\begin{thm}
\label{thm:duality}
Let $\Sigma \subseteq \Gr(l,\PP^n)$ be an irreducible hypersurface.
Then $\Sigma$ is coisotropic if and only if $\Sigma^\ast \subseteq \Gr(n-l-1,\PP^n)$ is coisotropic.
In that case, their underlying projective varieties are projectively dual to each other.
More precisely, under the canonical isomorphism $\Gr(n-l-1, (\PP^n)^\ast) = \Gr(l, \PP^n)$, we have 
\begin{align*}
\CH_i(X) &= \CH_{\dim X - \codim X^\vee+1 -i}(X^\vee), \text{ and} \\
\CH_i(X)^\ast = \CH_i(X^\ast) &= \CH_{\dim X - \codim X^\vee+1 -i}((X^\vee)^\ast),
\end{align*}
where $X \subseteq \mathbb{P}^n$ is the irreducible variety such that $\Sigma = \CH_i(X)$ for $i := \dim X+l+1-n$.
\end{thm}

\begin{proof}
The equality $\CH_i(X)^\ast = \CH_i(X^\ast)$ follows immediately from the definitions.
Hence, we only have to show $\CH_i(X) = \CH_{\dim X - \codim X^\vee+1 -i}(X^\vee)$ to conclude the proof.

For a point $x \in \PP^n$ and a hyperplane $H \subseteq \PP^n$ containing $x$, we define $\mathcal{G}_{x,H}(l,\PP^n) := \lbrace L \in \Gr(l,\PP^n) \mid x \in L \subseteq H \rbrace$.
Under the canonical isomorphism of Grassmannians, we have $\mathcal{G}_{H,x}(n-l-1, (\PP^n)^\ast) = \mathcal{G}_{x,H}(l,\PP^n)$.
Moreover, we consider the following open subset of the conormal variety: $\mathcal{U}_{X,X^\vee} := \lbrace (x,H) \in \mathrm{Reg}(X) \times \mathrm{Reg}(X^\vee) \mid T_xX \subseteq H \rbrace$.
By biduality, we have that $\mathcal{U}_{X^\vee,X} = \lbrace (H,x) \mid (x,H) \in \mathcal{U}_{X,X^\vee} \rbrace$.
Since the condition $\dim(L \cap T_xX) \geq i$ in the definition of $\CH_i(X)$ is equivalent to $\dim(L + T_xX) \leq n-1$, which means that $L$ is contained in a tangent hyperplane at $x$, we get
\begin{align*}
\CH_i(X) &= \overline{\bigcup_{(x,H) \in \mathcal{U}_{X,X^\vee}} \mathcal{G}_{x,H} \left(l, \PP^n\right)} \\
&= \overline{\bigcup_{(H,x) \in \mathcal{U}_{X^\vee,X}} \mathcal{G}_{H,x}\left(n-l-1, (\PP^n)^\ast \right)}
= \CH_{\dim X - \codim X^\vee+1 -i} \left(X^\vee \right).
\end{align*}
\end{proof}

Note that Theorems~\ref{thm:degrees} and~\ref{thm:duality} give an alternate proof that the polar degrees of a projective variety and its dual are the same but in reversed order, as stated in the forth property of polar degrees in Section \ref{sec:degree}.

\begin{cor}(Dual Cayley Trick)
\label{cor:dual_cayley}
$\quad$\\ Using the projection $q: S(n-\dim X+i,n+1) \to \Gr(n-\dim X+i-1, \PP^n), \; A \mapsto \mathrm{proj}(\mathrm{rs}\, A)$
sending full rank matrices to their projective row space, we have
\begin{align*}
\overline{q^{-1} \left( \CH_i(X) \right)}
= \overline{p^{-1} \left( \CH_{\dim X - \codim X^\vee +1 -i} \left( (X^\vee)^\ast \right) \right)}
= \left( (\PP^{n-k+i-1})^\ast \times X^\vee \right)^\vee.
\end{align*}
\end{cor}

\begin{proof}
This follows from Proposition~\ref{prop:cayley_trick} and Theorem~\ref{thm:duality}.
\end{proof}

\begin{ex}
The Chow form of a curve of degree at least two in $\mathbb{P}^3$ is the Hurwitz form of its dual surface.
Thus the variety of the orthogonal complements of lines intersecting the curve is the variety of lines tangent to the dual surface embedded in $\PP^3$.
For example, the dual of the twisted cubic is the surface cut out by the discriminant of a cubic univariate polynomial. The Chow form of the twisted cubic in primal Pl\"ucker coordinates is the determinant of the Bezout matrix $B$:
\begin{align*}
B := \left( \begin{array}{ccc}
 p_{01}&p_{02}&p_{03} \\ 
 p_{02}&p_{03}+p_{12}&p_{13} \\ 
 p_{03}&p_{13}&p_{23}
\end{array} \right) 
\leftrightsquigarrow
\left( \begin{array}{ccc}
 q_{23}&-q_{13}&q_{12} \\ 
 -q_{13}&q_{03}+q_{12}&-q_{02} \\ 
 q_{12}&-q_{02}&q_{01}
\end{array} \right) ,
\end{align*}
and the Hurwitz form of the discriminant surface in primal Pl\"ucker coordinates is the determinant of the matrix on right. \hfill $\diamondsuit$
\end{ex}

\begin{ex}
\label{ex:hurwitzHurwitz}
The Hurwitz form of a general surface in $\mathbb{P}^3$ is the Hurwitz form of its dual surface.
Consider for example the self-dual Segre-variety $\mathbb{P}^1 \times \mathbb{P}^1$. Its Hurwitz form is the determinant of the matrix $\left( \begin{smallmatrix}
2 p_{02} & p_{12}+p_{03} \\ p_{12}+p_{03} & 2p_{13}
\end{smallmatrix} \right)$, which stays invariant with respect to the change of coordinates \eqref{eq:plucker_change}.

This phenomenon was also observed in Example~\ref{ex:runningFermatDual}: Theorem~\ref{thm:duality} explains why the Hurwitz form of the Fermat cubic surface $X \subseteq \mathbb{P}^3$ and the Hurwitz form of its dual surface $X^\vee$ of degree 12 agree, and why the second coisotropic hypersurface of $X$ is exactly $X^\vee$. Moreover, it follows that the second coisotropic hypersurface of $X^\vee$ is the Fermat cubic surface $X$. \hfill $\diamondsuit$
\end{ex}

\begin{ex}
In $\Gr(2,\PP^4)$, there are three cases for coisotropic hypersurfaces: Chow forms of curves, Hurwitz forms of surfaces, and second coisotropic forms of threefolds. On the other hand, there are just two cases for $\Gr(1,\PP^4)$, namely Chow forms of surfaces and Hurwitz forms of threefolds. The following table summarizes which forms coincide, depending on the dimensions of the variety $X$ and its dual $X^\vee$ (Chow forms of threefolds are omitted).
\begin{align*}
\begin{array}{r|c|c|c}
X \mid X^\vee & \text{curve} & \text{surface} & \text{threefold} \\ 
\hline
\text{curve} & & \CH_0(X) = \CH_0(X^\vee) & \CH_0(X) = \CH_1(X^\vee) \\ 
\hline
\text{surface} & \CH_0(X) = \CH_0(X^\vee) & \CH_0(X) = \CH_1(X^\vee) & \CH_0(X) = \CH_2(X^\vee) \\ 
 &  & \CH_1(X) = \CH_0(X^\vee) & \CH_1(X) = \CH_1(X^\vee) \\ 
\hline
\text{threefold} & \CH_1(X) = \CH_0(X^\vee) & \CH_1(X) = \CH_1(X^\vee) & \CH_1(X) = \CH_2(X^\vee) \\
 &  & \CH_2(X) = \CH_0(X^\vee) & \CH_2(X) = \CH_1(X^\vee)
\end{array}
\end{align*}

\vspace*{-9mm}
\hfill $\diamondsuit$
\end{ex}

\section{Hyperdeterminants Revisited}
\label{sec:hyperdet}
The purpose of this section is to derive and discuss the following result.

\begin{prop}
\label{prop:hyperdet}
The $i$-th coisotropic form of the Segre variety $\mathbb{P}^{n_1} \times \ldots \times \mathbb{P}^{n_d}$ in \linebreak[4] $\mathbb{P}^{(n_1+1)\cdots(n_d+1)-1}$, in primal Stiefel coordinates, coincides with the hyperdeterminant of \linebreak[4] format $(n_1 + \ldots + n_d-i+1) \times (n_1+1) \times \ldots \times (n_d+1)$. All hyperdeterminants arise in that manner.
\end{prop}

Chapter 14 of \cite{gkz} is devoted to the study of hyperdeterminants. They are defined as follows.
For $n_1, \ldots, n_d \geq 1$, the variety $ X:=\mathbb{P}^{n_1} \times \ldots \times \mathbb{P}^{n_d}$ characterizes all tensors of format $(n_1+1) \times \ldots \times (n_d+1)$ having rank at most one. Whenever the dual variety~$X^\vee$ is a hypersurface, its defining polynomial is called the \emph{hyperdeterminant of format \mbox{$(n_1+1) \times \ldots \times (n_d+1)$}}.
Analogously to Corollary~\ref{cor:dim}, one can derive the condition for $\codim X^\vee=1$:
recall that $\mu(Y) := \dim Y + \codim Y^\vee-1$ for every irreducible variety $Y \subseteq \mathbb{P}^n$. The equality \eqref{eq:product_thm} proven in \cite[Ch.~1, Thm.~5.5]{gkz} generalizes by induction to
\begin{align*}
\mu(X_1 \times \ldots \times X_d) = \max \lbrace \dim X_1 + \ldots + \dim X_d, \mu(X_1), \ldots, \mu(X_d) \rbrace.
\end{align*}
Hence $X^\vee$ is a hypersurface if and only if $2n_i \leq n_1 + \ldots + n_d$ for all $i = 1, \ldots, d$, and
\begin{align*}
\codim X^\vee = \max \left\lbrace 1, 2 \max \lbrace n_1, \ldots, n_d \rbrace - (n_1 + \ldots + n_d)+1 \right\rbrace.
\end{align*}

\begin{ex}
In the special case $d=2$ of matrices, the projectively dual variety $X^\vee$ is given by all matrices that do not have full rank. This variety is a hypersurface if and only if the matrices have square format $(n_1=n_2)$. In this case, the defining polynomial of $X^\vee$ is the usual determinant. Otherwise the codimension of $X^\vee$ equals $|n_2-n_1|+1$.
\hfill $\diamondsuit$
\end{ex}

\begin{proof}[Proof of Proposition~\ref{prop:hyperdet}]
Let $0 \leq i \leq n_1 + \ldots + n_d -\codim X^\vee +1$. By the Cayley trick in Proposition \ref{prop:cayley_trick}, the $i$-th coisotropic form of $X$ written in primal Stiefel coordinates is exactly the hyperdeterminant of format $(n_1 + \ldots + n_d-i+1) \times (n_1+1) \times \ldots \times (n_d+1)$. It is clear that all hyperdeterminants arise in that way as coisotropic forms of the varieties of tensors with rank at most one.
\end{proof}

\begin{rem}
Note that even the usual determinant of square matrices is given by the Chow form of $\mathbb{P}^n$.
Using the duality explained in Theorem~\ref{thm:duality} and Corollary~\ref{cor:dual_cayley}, the hyperdeterminants can also be characterized as the coisotropic forms (written in dual Stiefel coordinates) of the varieties of \emph{degenerate tensors}. \hfill $\diamondsuit$
\end{rem}

If all inequalities $2 n_i \leq n_1 + \ldots + n_d$ are satisfied (which means that $\codim X^\vee = 1$) such that at least one of them holds with equality, the hyperdeterminant is said to be of \emph{boundary format}. An example for this is the determinant of square matrices.  
The hyperdeterminants of boundary format can also be characterized in terms of coisotropic forms.
This is also studied in \cite[Ch.~14, Sec.~3C]{gkz}, but here this naturally and immediately follows from the duality studied in Theorem~\ref{thm:duality}.

\begin{cor}
The Chow form of the Segre variety $X = \mathbb{P}^{n_1} \times \ldots \times \mathbb{P}^{n_d}$ in primal Stiefel coordinates is a hyperdeterminant of boundary format, and~-- up to permuting the tensor format~-- all hyperdeterminants of boundary format arise in that manner.

If $\codim X^\vee \geq 2$, then the Chow form of $(X^\vee)^\ast$ in dual Stiefel coordinates is a hyperdeterminant of boundary format, and~-- up to permuting the tensor format~-- all hyperdeterminants of boundary format arise in that manner.
\end{cor}

\begin{proof}
The first part of this proposition is clear.
Note that the second part uses the convention that the Chow form of the empty variety $(\mathbb{P}^n)^\vee$ in dual Stiefel coordinates is the usual $(n+1)\times(n+1)$-determinant. 
If $X^\vee$ is not a hypersurface and $d \geq 2$, exactly two coisotropic forms of $X$ yield hyperdeterminants of boundary format:
the Chow form and the $(2 \cdot (n_1 + \ldots + n_d - \max \lbrace n_1, \ldots, n_d \rbrace))$-th coisotropic form, where~-- by Theorem~\ref{thm:duality}~-- the latter is the Chow form of $X^\vee$.
Hence, although $X^\vee$ is not defining a hyperdeterminant, the Chow form of $(X^\vee)^\ast$ in dual Stiefel coordinates coincides with the hyperdeterminant of boundary format $(2 \max \lbrace n_1, \ldots, n_d \rbrace - (n_1 + \ldots +n_d) +1) \times (n_1+1) \times \ldots \times (n_d+1)$. 
\end{proof}

\begin{rem}
Theorem 3.3 in Chapter 14 of \cite{gkz} shows that all hyperdeterminants of boundary format can be written as the usual determinant of a square matrix whose entries are linear forms in the tensor entries.
Hence, if $X^\vee$ is not a hypersurface, the Chow forms of $X$ and $X^\vee$ have determinantal representations in their Stiefel coordinates.

Analogously, if $X^\vee$ is a hypersurface and the corresponding hyperdeterminant is of boundary format, the Chow form of $X$ and the $(n_1 + \ldots + n_d)$-th coisotropic hypersurface of $X$ (which is just $X^\vee$) give hyperdeterminants of boundary format. These are the only two coisotropic hypersurfaces of $X$ with that property, and their defining polynomials in Stiefel coordinates have determinantal representations. Finally, if $X^\vee$ is a hypersurface and its hyperdeterminant is not of boundary format, the Chow form of $X$ is the only coisotropic form which yields a hyperdeterminant of boundary format. In all cases the Chow form of $X$ has a determinantal representation in primal Stiefel coordinates. \hfill $\diamondsuit$
\end{rem}

\begin{ex}
The variety $X = \mathbb{P}^1 \times \mathbb{P}^n$ of $2 \times (n+1)$-matrices of rank at most one is self-dual and it has three coisotropic hypersurfaces.
Hence, after the change of coordinates in \eqref{eq:plucker_change}, the Chow form of $X$ is the same as the second coisotropic form of $X$.

For $n=1$, the variety $X$ itself is a hypersurface, given by the $2 \times 2$-determinant. Therefore, its Chow form and its second coisotropic form are also the $2 \times 2$-determinant in their respective Pl\"ucker coordinates. As mentioned in Example~\ref{ex:hurwitzHurwitz}, the Hurwitz form of~$X$ is the determinant of $\left( \begin{smallmatrix}
2 p_{02} & p_{12}+p_{03} \\ p_{12}+p_{03} & 2p_{13}
\end{smallmatrix} \right)$, which leads after substitution by the $2 \times 2$-minors of a general $2 \times 4$-matrix to the hyperdeterminant of format $2 \times 2 \times 2$.

Analogously, the Chow form of $X$ in primal Stiefel coordinates is the hyperdeterminant of boundary format $3 \times 2 \times 2$. Hence, this hyperdeterminant can be written as the determinant of $\left(\begin{smallmatrix} p_{012} & p_{013} \\ p_{023} & p_{123} \end{smallmatrix}\right)$, where the $p_{ijk}$ are the $3 \times 3$-minors of a general $3 \times 4$-matrix. On the other hand, the $3 \times 2 \times 2$-hyperdeterminant has a determinantal representation: let $A,B \in \mathbb{C}^{3 \times 2}$ be the two slices of a general $3 \times 2 \times 2$-tensor in the last direction. The $3 \times 2 \times 2$-hyperdeterminant is the determinant of the $6 \times 6$-matrix $\left(\begin{smallmatrix}
A & B & 0 \\ 0 & A & B
\end{smallmatrix}\right)$, since by Laplace expansion in the first three rows this determinant is equal to $p_{012}p_{123}-p_{013}p_{023}$, where the $p_{ijk}$ are the minors of the $3 \times 4$-matrix $(A \; B)$.
Moreover, the $2 \times 2 \times 3$-hyperdeterminant is also given by the second coisotropic form of $\mathbb{P}^1 \times \mathbb{P}^2$ in primal Stiefel coordinates, or equivalently by the Chow form of $\mathbb{P}^1 \times \mathbb{P}^2$ in dual Stiefel coordinates. Thus, the $2 \times 2 \times 3$-hyperdeterminant can be obtained by substituting the $2 \times 2$-minors of the general $6 \times 2$-matrix $\left( \begin{smallmatrix} A \\ B \end{smallmatrix} \right)$ into the Chow form
\begin{align*}
&q_{13}q_{14}q_{24}-q_{03}q_{14}q_{25}-q_{12}q_{14}q_{25}+q_{02}q_{15}q_{25}-q_{03}q_{14}q_{34}+q_{01}q_{34}^2+q_{03}q_{04}q_{35}\\-&q_{02}q_{05}q_{35}+q_{02}q_{14}q_{35}-q_{01}q_{24}q_{35}+q_{12}^2q_{45}-q_{02}q_{13}q_{45}-q_{01}q_{23}q_{45}.
\hskip \textwidth minus \textwidth \diamondsuit
\end{align*}

\vspace*{-9mm}
\hfill $\diamondsuit$
\end{ex}

\section{The Cayley Variety}
\label{sec:pluecker}
Chow forms of space curves and Hurwitz forms of surfaces in $\mathbb{P}^3$~-- which are all cases of coisotropic hypersurfaces in $\Gr(1,\PP^3)$~-- were already studied by Cayley \cite{cayley2}.
Therefore the variety $\mathcal{C}(l,d,\PP^n)$ of all coisotropic forms of degree $d$ in the coordinate ring of $\Gr(l,\PP^n)$ is called \emph{Cayley variety} in the following.
Its subvariety of all Chow forms was introduced by Chow and van der Waerden~\cite{chow} and is called \emph{Chow variety}.
The problem of recognizing the Chow forms among all coisotropic forms is addressed in \cite[Ch.~4, Sec.~3]{gkz}.
This goes already back to Green and Morrison~\cite{green_morr}, who gave explicit equations for the Chow variety.

In \cite{our_computation}, the vanishing ideals of the Cayley variety and the Chow variety of hypersurfaces of degree two in $\Gr(1,\PP^3)$ were computed.
It was found that $\mathcal{C}(1,2,\PP^3)$ has degree $92$ and codimension $10$ in a $19$-dimensional projective space. Moreover, the prime decomposition of its radical ideal into Chow forms of curves and Hurwitz forms of surfaces is explicitly computed. In particular, this decomposition led to the interesting observation that every Chow form of a plane conic in $\mathbb{P}^3$ is already contained in the Zariski closure of all Hurwitz forms of surfaces in $\mathbb{P}^3$.

In contrast to the previous sections of this article, where hypersurfaces in Grassmannians were mainly studied in affine or Stiefel coordinates, the computation in \cite{our_computation} used a differential characterization of coisotropy in Pl\"ucker coordinates, which goes also back to Cayley.
He had already realized that a hypersurface $\Sigma \subseteq \Gr(1,\PP^3)$ with defining polynomial $Q$ in Pl\"ucker coordinates is coisotropic if and only if the following polynomial vanishes everywhere on $\Sigma$:
\begin{align*}
	\forall L \in \Sigma: 
	\left(\frac{\partial Q}{\partial p_{01}} \cdot \frac{\partial Q}{\partial p_{23}}
	- \frac{\partial Q}{\partial p_{02}} \cdot \frac{\partial Q}{\partial p_{13}}
	+ \frac{\partial Q}{\partial p_{03}} \cdot \frac{\partial Q}{\partial p_{12}} \right)(L)= 0.
\end{align*}
This follows immediately from the affine or Stiefel characterization of coisotropy given in Propositions \ref{prop:affine_characterization} and \ref{prop:stiefel_characterization}. 
In this section, we provide a generalization of Cayley's result to $\Gr(1,\PP^n)$ for $n \geq 3$, and describe how this can be used to compute the vanishing ideal of the Cayley variety $\mathcal{C}(1,d,\PP^n)$.

For a homogeneous irreducible polynomial $Q$ in dual Pl\"ucker coordinates of $\Gr(1,\PP^n)$ and for $0 \leq j < k < i < m \leq n$, define
\begin{align*}
R^Q_{jkim} := \frac{\partial Q}{\partial q_{jk}} \cdot \frac{\partial Q}{\partial q_{im}}
        - \frac{\partial Q}{\partial q_{ji}} \cdot \frac{\partial Q}{\partial q_{km}}
        + \frac{\partial Q}{\partial q_{jm}} \cdot \frac{\partial Q}{\partial q_{ki}}.
\end{align*}
Note that this polynomial of degree $2(\deg Q - 1)$ is a differential version of the usual Pl\"ucker relations. To allow permutations of the indices, define $R^Q_{\pi(jkim)} := \mathrm{sgn}(\pi)R^Q_{jkim}$ for $\pi \in S_4$.

\begin{thm}
\label{thm:plueckerChar}
Let $n \geq 3$, and let $\Sigma \subseteq \Gr(1,\PP^n)$ be an irreducible hypersurface, given by a homogeneous polynomial $Q$ in dual Pl\"ucker coordinates. Then $\Sigma$ is coisotropic if and only if for all $0 \leq i < m \leq n$, the following polynomial in dual Pl\"ucker coordinates vanishes everywhere on $\Sigma$:
\begin{align*}
	\forall L \in \Sigma: \left( \sum \limits_{0 \leq j < k \leq n, j,k \notin \lbrace i,m \rbrace}
        q_{jk} R^Q_{jkim} \right) (L) = 0.
\end{align*}
\end{thm}

\begin{proof}
	This proof relies on Proposition~\ref{prop:stiefel_characterization}.
	Let again $\pl$ be the map that sends a matrix to its maximal minors, such that $\pl(A)_{ij} = a_{0i}a_{1j}-a_{0j}a_{1i}$ denotes the minor given by the $i$-th and $j$-th column of a $2 \times (n+1)$-matrix $A$.
	For shorter notation, write $\beta_{ij}:= \frac{\partial Q}{\partial q_{ij}}(\pl(\cdot))$, as well as $q_{ij} = \pl(\cdot)_{ij}$. This will be used to compute the $(2 \times 2)$-minors of $J_{Q \circ \pl}$. For this, pick two columns with indices $0 \leq i<m \leq n$. The set of remaining column indices is denoted by $S := \lbrace 0, \ldots, n \rbrace \setminus \lbrace i,m \rbrace$. The chain rule for partial derivatives gives
\begin{align*}
	\frac{\partial (Q \circ \pl)}{\partial a_{0i}} = - \sum \limits_{j \in S} \beta_{ji} a_{1j} - \beta_{mi}a_{1m}, \quad\quad\quad 
	\frac{\partial (Q \circ \pl)}{\partial a_{1m}} = \sum \limits_{j \in S} \beta_{jm} a_{0j} + \beta_{im}a_{0i}.
\end{align*}
Hence, the $(2 \times 2)$-minor of $J_{Q \circ \pl}$ given by the columns $i$ and $m$ equals
\begin{align*}
	&\frac{\partial (Q \circ \pl)}{\partial a_{0i}} \cdot \frac{\partial (Q \circ \pl)}{\partial a_{1m}} - \frac{\partial (Q \circ \pl)}{\partial a_{1i}} \cdot \frac{\partial (Q \circ \pl)}{\partial a_{0m}} \\
	= &a_{0i}a_{1m}\beta_{im}^2 + \sum \limits_{j \in S} a_{0j}a_{1m} \beta_{im}\beta_{jm} - \sum \limits_{j \in S} a_{0i}a_{1j}\beta_{im}\beta_{ji}
	- \sum \limits_{j,k \in S}a_{0k}a_{1j}\beta_{ji}\beta_{km} \\
	&-a_{0m}a_{1i}\beta_{im}^2 + \sum \limits_{j \in S} a_{0j}a_{1i} \beta_{im}\beta_{ji} - \sum \limits_{j \in S} a_{0m}a_{1j}\beta_{im}\beta_{jm}
	+ \sum \limits_{j,k \in S}a_{0j}a_{1k}\beta_{ji}\beta_{km} \\
	=&q_{im}\beta_{im}^2 + \sum \limits_{j \in S} q_{jm} \beta_{im}\beta_{jm} 
	+ \sum \limits_{j \in S} q_{ji}\beta_{im}\beta_{ji}
	+ \sum \limits_{j,k \in S, j \neq k} q_{jk} \beta_{ji}\beta_{km} \\
	=&\beta_{im} \left( \sum \limits_{0 \leq j < k \leq n} q_{jk}\beta_{jk} - \sum \limits_{j,k \in S, j < k} q_{jk}\beta_{jk} \right)
	+ \sum \limits_{j,k \in S, j < k}q_{jk} \beta_{ji}\beta_{km}
	- \sum \limits_{j,k \in S, j < k}q_{jk} \beta_{ki}\beta_{jm}\\
	=& \beta_{im} \sum \limits_{0 \leq j < k \leq n} q_{jk}\beta_{jk}  -\sum \limits_{j,k \in S, j < k} q_{jk} \left( \beta_{jk}\beta_{im} - \beta_{ji}\beta_{km}+\beta_{jm}\beta_{ki} \right)\\
	=& \beta_{im} \sum \limits_{0 \leq j < k \leq n} q_{jk}\beta_{jk}-\sum \limits_{j,k \in S, j < k} q_{jk} R^Q_{jkim}.
\end{align*}
Since $Q$ is homogeneous, we have for all $N \in V(Q \circ \pl)$ that $ \sum \limits_{0 \leq j < k \leq n} (q_{jk}\beta_{jk})(N)=0 $.
Now the theorem follows from Proposition \ref{prop:stiefel_characterization}.
\end{proof}

Note that for $n=3$, the above theorem yields exactly Cayley's differential characterization: $R^Q_{0123}(L) = 0$ for all $L \in \Sigma$.
For $n=4$, one gets the following 10 polynomials in dual Pl\"ucker coordinates:
\begin{align}
\label{eq:diffCayleyG(2,5)}
\begin{split}
&q_{01}R_{0134}^Q + q_{02}R_{0234}^Q + q_{12}R_{1234}^Q, \quad\quad\quad \;\;\; q_{01}R_{0124}^Q - q_{03}R_{0234}^Q - q_{13}R_{1234}^Q,\\
&q_{01}R_{0123}^Q + q_{04}R_{0234}^Q + q_{14}R_{1234}^Q, \quad\quad\quad -q_{02}R_{0124}^Q - q_{03}R_{0134}^Q + q_{23}R_{1234}^Q,\\
-&q_{02}R_{0123}^Q + q_{04}R_{0134}^Q - q_{24}R_{1234}^Q, \quad\quad\quad \;\;\; q_{03}R_{0123}^Q + q_{04}R_{0124}^Q + q_{34}R_{1234}^Q,\\
&q_{12}R_{0124}^Q + q_{13}R_{0134}^Q + q_{23}R_{0234}^Q, \quad\quad\quad \;\;\; q_{12}R_{0123}^Q - q_{14}R_{0134}^Q - q_{24}R_{0234}^Q,\\
-&q_{13}R_{0123}^Q - q_{14}R_{0124}^Q + q_{34}R_{0234}^Q, \quad\quad\quad \;\;\; q_{23}R_{0123}^Q + q_{24}R_{0124}^Q + q_{34}R_{0134}^Q.
\end{split}
\end{align}

Theorem~\ref{thm:plueckerChar} gives a method to compute the vanishing ideal of the Cayley variety $\mathcal{C}(1,d,\PP^n)$.
For positive integers $N$ and $D$, denote by $\left( \left( \begin{smallmatrix} N \\ D \end{smallmatrix} \right)\right) :=  \left( \begin{smallmatrix} N+D-1 \\ D \end{smallmatrix} \right)$ the multiset coefficient, i.e., the number of monomials of degree $D$ in $N$ variables.
\begin{cor}
	Consider the Cayley variety $\mathcal{C}(1,d,\PP^n) \subseteq \mathbb{P} (\mathbb{C}[\Gr(1,\PP^n)]_d)$, and let $\bs{c}$ be a vector of homogeneous coordinates on $\mathbb{P} (\mathbb{C}[\Gr(1,\PP^n)]_d)$.
	There are $ \left( \begin{smallmatrix} n+1 \\ 2 \end{smallmatrix} \right)$ matrices of size
	\begin{align}
		\label{eq:matrix_size}
		\left( \left( \begin{array}{c} \left( \begin{array}{c} n+1 \\ 2  \end{array} \right) \\ 2d-1 \end{array} \right) \right)
		\times 
		\left[1+  \left( \left( \begin{array}{c} \left( \begin{array}{c} n+1 \\ 2  \end{array} \right) \\ d-1 \end{array} \right) \right) +
		\left( \begin{smallmatrix} n+1 \\ 4 \end{smallmatrix} \right) \cdot \left( \left( \begin{array}{c} \left( \begin{array}{c} n+1 \\ 2  \end{array} \right) \\ 2d-3 \end{array} \right) \right) \right]
	\end{align}
	whose entries are polynomials in $\bs{c}$, such that the ideal generated by the maximal minors of these matrices defines~-- up to saturation~-- the Cayley variety $\mathcal{C}(1,d,\PP^n)$.
	Moreover, these minors have degree 
	\begin{align}
		\label{eq:minors_degree}
		2+ \left( \left( \begin{array}{c} \left( \begin{array}{c} n+1 \\ 2  \end{array} \right) \\ d-1 \end{array} \right) \right)
	\end{align}
	in the $\dim(\mathbb{C}[\Gr(1,\PP^n)]_d)$ many unknowns $\bs{c}$.
\end{cor}

\begin{proof}
	Let $Q$ be a general homogeneous polynomial of degree $d$ in the dual Pl\"ucker coordinates $q_{ij}$ of $\Gr(1,\PP^n)$. Denote the coefficient vector of the polynomial $Q$ by $\bs{c}$. The entries of $\bs{c}$ serve as homogeneous coordinates on $ \mathbb{P} (\mathbb{C}[\Gr(1,\PP^n)]_d)$, although~-- due to the Pl\"ucker relations~-- they are not independent unknowns.
The characterization in Theorem~\ref{thm:plueckerChar} states that the equation
\begin{align*}
C_{i,m} := \sum \limits_{0 \leq j < k \leq n, j,k \notin \lbrace i,m \rbrace}
        q_{jk} R^Q_{jkim}
\end{align*}
vanishes everywhere on the hypersurface of $\Gr(1,\PP^n)$ defined by $Q$, for all $0 \leq i < m \leq n$. Equivalently, the polynomial $C_{i,m}$ is contained in the radical of the ideal generated by $Q$ and the Pl\"ucker relations. Under the assumption that this ideal is already radical, we get the condition 
\begin{align}
\label{eq:cayleyCondition}
C_{i,m} - F^{(d-1)} \cdot Q - \sum \limits_{0 \leq \alpha < \beta < \gamma < \delta \leq n} G^{(2d-3)}_{\alpha \beta \gamma \delta} \cdot \mathcal{R}_{\alpha \beta \gamma \delta} =0,
\end{align}
where $\mathcal{R}_{\alpha \beta \gamma \delta}$ denotes the quadratic Pl\"ucker relation $q_{\alpha \beta}q_{\gamma \delta}-q_{\alpha \gamma}q_{\beta \delta} + q_{\alpha \delta}q_{\beta \gamma}$, and $G^{(2d-3)}_{\alpha \beta \gamma \delta}$ and $F^{(d-1)}$ are homogeneous polynomials of degree $2d-3$ and $d-1$, respectively. Let~$\bs{a}$ denote the coefficient vector of $F^{(d-1)}$, and let $\bs{b}$ denote the vector of all coefficients of all $G^{(2d-3)}_{\alpha \beta \gamma \delta}$.
The coefficient of each monomial of \eqref{eq:cayleyCondition}, where the variables are the Pl\"ucker coordinates $q_{ij}$, has quadratic terms in $\bs{c}$ (coming from $C_{i,m}$), multilinear terms in $\bs{a}$ and~$\bs{c}$ (coming from $F^{(d-1)} \cdot Q$), and linear terms in $\bs{b}$ (coming from $\sum G^{(2d-3)}_{\alpha \beta \gamma \delta} \cdot \mathcal{R}_{\alpha \beta \gamma \delta}$). Hence, we can represent such a coefficient as a vector:
the quadratic terms in $\bs{c}$ are the first entry. For each coefficient in $\bs{a}$ we add an entry, namely the corresponding linear form in~$\bs{c}$. Finally, we add the constant factor of each coefficient in $\bs{b}$.

In this way, we get a vector for each monomial in \eqref{eq:cayleyCondition}. 
Let $M_{i,m}$ be the matrix whose rows are given by these vectors. To sum up, the rows of $M_{i,m}$ are indexed by the monomials of \eqref{eq:cayleyCondition}, its columns are indexed by the entries in the vector $(1,\bs{a},\bs{b})$, and the entries of $M_{i,m}$ are (at most quadratic) polynomials in $\bs{c}$:
the first column contains quadrics, the columns corresponding to $\bs{a}$ consist of linear forms, and the remaining columns (corresponding to $\bs{b}$) contain constants.
In particular, the matrix $M_{i,m}$ has size~\eqref{eq:matrix_size}, and its maximal minors have degree~\eqref{eq:minors_degree} in $\bs{c}$.
The condition \eqref{eq:cayleyCondition} is equivalent to that the vector $(1,\bs{a},\bs{b})$ is contained in the kernel of the matrix $M_{i,m}$. Hence, for all $0 \leq i < m \leq n$, the maximal minors of $M_{i,m}$ give basic equations for the vanishing ideal of the Cayley variety $\mathcal{C}(1,d,\PP^n)$, but one still has to do some careful computational work to compute the actual vanishing ideal. There are three reasons for this.
First, the maximal minors of~$M_{i,m}$ also capture vectors in the kernel of $M_{i,m}$ that are of the form $(0,\bs{a},\bs{b})$. Thus one still has to saturate by minors of the matrix that is obtained by deleting the first column from $M_{i,m}$.
Secondly, we assumed the ideal generated by $Q$ and the Pl\"ucker relations to be radical. Therefore, this method might not characterize all coisotropic forms. 
Finally, the maximal minors of the matrices $M_{i,m}$ already lead to extraneous factors that arise since Theorem~\ref{thm:plueckerChar} requires $Q$ to be irreducible. In particular, all squares trivially satisfy the condition \eqref{eq:cayleyCondition}.
\end{proof}

\begin{ex}
The above method was explicitly computed in \cite{our_computation} for the Cayley variety $\mathcal{C}(1,2,\PP^3)$ of quadratic coisotropic forms in $\Gr(1,\PP^3)$. In this case, condition \eqref{eq:cayleyCondition} reduces to
$R_{0123}^Q - s \cdot Q - t \cdot \mathcal{R}_{0123}$,
for constants $s$ and $t$. We get only one matrix $M$ with~$3$ columns and $21$ rows. This matrix is given explicitly in Figure 1 of \cite{our_computation} (but with columns in reversed order as described here).
In this case, we do not need to compute saturations, since it cannot happen that the kernel of $M$ contains vectors of the form $(0,\bs{a},\bs{b})$. By Proposition~1 of \cite{our_computation}, the $3 \times 3$-minors of $M$ form the vanishing ideal of the Cayley variety $\mathcal{C}(1,2,\PP^3)$, up to the extraneous factor of all quadrics that are squares modulo the Pl\"ucker relation.
These are $\left( \begin{smallmatrix} 21 \\ 3 \end{smallmatrix} \right)$ equations of degree 3 in the 21 unknowns $\bs{c}$, which are in fact just $20= \dim \mathbb{C}[\Gr(1,\PP^3)]_2$ unknowns due to the Pl\"ucker relation.
\hfill $\diamondsuit$
\end{ex}

\begin{ex}
Consider the Cayley variety $\mathcal{C}(1,3,\PP^4)$ of cubic coisotropic forms in $\Gr(1,\PP^4)$. We have the ten equations in~\eqref{eq:diffCayleyG(2,5)} of degree 5 in the 10 variables $q_{ij}$. For each such equation, the condition~\eqref{eq:cayleyCondition} contains 2002 monomials. The quadric $F^{(2)}$ has 55 monomials and the cubics $G_{\alpha \beta \gamma \delta}^{(3)}$ have 220 monomials each. This leads to ten matrices with 2002 rows and $1156 = 1+55+5 \cdot 220$ columns. The first column of each matrix consists of quadratic forms in $\bs{c}$, the next 55 columns contain linear form in $\bs{c}$, and the remaining columns have only constants. The maximal minors of these matrices are thus $10 \cdot \left( \begin{smallmatrix} 2002 \\ 1156 \end{smallmatrix} \right)$ equations of degree 57 in the 220 unknowns $\bs{c}$, which are in fact just $175 = \dim \mathbb{C}[\Gr(1,\PP^4)]_3$ unknowns due to the Pl\"ucker relations. Hence, the computation of the vanishing ideal of the Cayley variety $\mathcal{C}(1,3,\PP^4)$ is a hard computational task.
\hfill $\diamondsuit$
\end{ex}

\section{Computations}
\label{sec:computations}

A \texttt{Macaulay2} package for calculating coisotropic hypersurfaces and recovering their underlying varieties can be obtained at
\begin{center}
  \url{page.math.tu-berlin.de/~kohn/packages/Coisotropy.m2}
\end{center}

To use the package, the user can simply start \texttt{Macaulay2} from the same directory where the file was saved and then use the command \texttt{loadPackage "Coisotropy"}. After that, the following commands are available:

\texttt{dualVariety I}: Computes the ideal of the projectively dual variety of the projective variety given by the ideal \texttt{I}.

\texttt{polarDegrees I}: Computes a list whose $i$-th entry is the degree of the $i$-th coisotropic hypersurface of the projective variety given by the ideal \texttt{I}. This is done by computing the multidegree of the conormal variety, as described in Section~\ref{sec:degree}.

\texttt{coisotropicForm (I,i)}: Returns the i-th coisotropic form in primal Pl\"ucker coordinates of the projective variety given by the ideal \texttt{I}. The computation of this form follows essentially Definition~\ref{defn:higher_associated}.

\texttt{isCoisotropic (Q,k,n)}: Checks if a hypersurface in $\Gr(\mathtt{k},\PP^\mathtt{n})$ is coisotropic. The hypersurface is given by a polynomial \texttt{Q} in primal Pl\"ucker coordinates. This is implemented by using the characterization of coisotropy in Proposition~\ref{prop:affine_characterization}.

\texttt{recoverVar (Q,k,n)}: Computes the ideal of the underlying projective variety of a coisotropic hypersurface in $\Gr(\mathtt{k},\PP^\mathtt{n})$, which is given by a polynomial \texttt{Q} in primal Pl\"ucker coordinates. This computation uses the Cayley trick in Proposition~\ref{prop:cayley_trick}.

\texttt{dualToPrimal (Q,k,n)}: Transforms the polynomial \texttt{Q} in dual Pl\"ucker coordinates of $\Gr(\mathtt{k},\PP^\mathtt{n})$ to a polynomial in primal Pl\"ucker coordinates. This can be used to perform the change of coordinates~\eqref{eq:plucker_change} before calling one of the above commands that require primal Pl\"ucker coordinates.

\texttt{primalToDual (Q,k,n)}: Reverse transformation to \texttt{dualToPrimal}.

\paragraph*{Acknowledgments.}
During this work I was supported by a Fellowship from the Einstein Foundation Berlin.
I would like to thank Bernd Sturmfels, Peter Bürgisser, Pierre Lairez, Ragni Piene and Paolo Tripoli for helpful discussions and comments.

\bibliographystyle{alpha}
\bibliography{literatur}

\begin{thebibliography}{CvdW37}

\bibitem[BKLS16]{our_computation}
P.~Bürgisser, K.~Kohn, P.~Lairez, and B.~Sturmfels.
\newblock {Computing the Chow Variety of Quadratic Space Curves}.
\newblock In I.~Kotsireas, S.~Rump, and C.~Yap, editors, {\em Mathematical
  Aspects of Computer and Information Sciences, MACIS 2015, Berlin}, pages
  130--136, 2016.

\bibitem[BL]{schubert}
P.~Bürgisser and A.~Lerario.
\newblock {Probabilistic Schubert Calculus}.
\newblock arXiv:1612.06893.

\bibitem[Bü17]{condition}
P.~Bürgisser.
\newblock {Condition of intersecting a projective variety with a varying linear
  subspace}.
\newblock {\em SIAM Journal on Applied Algebra and Geometry}, 1(1):111--125,
  2017.

\bibitem[Cat14]{catanese}
F.~Catanese.
\newblock {Cayley Forms and Self-Dual Varieties}.
\newblock In {\em Proceedings of the Edinburgh Mathematical Society (Series
  2)}, volume~57, pages 89--109. Cambridge University Press, 2014.

\bibitem[Cay60]{cayley2}
A.~Cayley.
\newblock {On a new analytical representation of curves in space}.
\newblock {\em The Quarterly Journal of Pure and Applied Mathematics},
  3:225--236, 1860.

\bibitem[CvdW37]{chow}
W-L. Chow and B.L. van~der Waerden.
\newblock {Zur algebraischen Geometrie. IX. \"Uber zugeordnete Formen und
  algebraische Systeme von algebraischen Mannigfaltigkeiten}.
\newblock {\em Mathematische Annalen}, 113:696--708, 1937.

\bibitem[FKM83]{kleiman2}
William Fulton, Steven~L. Kleiman, and Robert MacPherson.
\newblock About the enumeration of contacts.
\newblock {\em Algebraic Geometry—Open Problems}, pages 156--196, 1983.

\bibitem[GKZ94]{gkz}
I.M. Gel'fand, M.M. Kapranov, and A.V. Zelevinsky.
\newblock {\em {Discriminants, Resultants and Multidimensional Determinants}},
  volume 227 of {\em Graduate Texts in Mathematics}.
\newblock Birkh\"auser, Boston, 1994.

\bibitem[GM86]{green_morr}
M.~Green and I.~Morrison.
\newblock {The equations defining Chow varieties}.
\newblock {\em Duke Mathematical Journal}, 53:733--747, 1986.

\bibitem[GS]{m2}
D.R. Grayson and M.E. Stillman.
\newblock {Macaulay2, a software system for research in algebraic geometry}.
\newblock Available at \texttt{http://www.math.uiuc.edu/Macaulay2/}.

\bibitem[Hol88]{holme}
A.~Holme.
\newblock {The geometric and numerical properties of duality in projective
  algebraic geometry}.
\newblock {\em Manuscripta mathematica}, 61:145--162, 1988.

\bibitem[Kle86]{kleiman1}
Steven~L. Kleiman.
\newblock Tangency and duality.
\newblock In {\em Proceedings of the 1984 Vancouver conference in algebraic
  geometry}, volume~6, pages 163--225, 1986.

\bibitem[KST]{vision}
K.~Kohn, B.~Sturmfels, and M.~Trager.
\newblock {Changing Views on Curves and Surfaces}.
\newblock arXiv:1707.01877.

\bibitem[Pie78]{piene}
R.~Piene.
\newblock {Polar classes of singular varieties}.
\newblock {\em Annales Scientifiques de l’\'Ecole Normale Sup\'erieure},
  11:247--276, 1978.

\bibitem[Stu17]{hurwitz}
B.~Sturmfels.
\newblock {The Hurwitz Form of a Projective Variety}.
\newblock {\em Journal of Symbolic Computation}, 79:186--196, 2017.

\end{thebibliography}

\end{document}